\documentclass[11pt]{article}
\usepackage{amsmath}
\usepackage{amsfonts} \makeatother
\usepackage{amsfonts,amsmath,amssymb,mathrsfs}
\usepackage{cite}
\usepackage{amsthm}
\usepackage[colorlinks,bookmarksopen,bookmarksnumbered,citecolor=red, linkcolor=red, urlcolor=red]{hyperref}

\numberwithin{equation}{section}
\newtheorem{theorem}{Theorem}[section]
\newtheorem{lemma}[theorem]{Lemma}
\newtheorem{proposition}[theorem]{Proposition}
\newtheorem{remark}{Remark}[section]

\newtheorem{corollary}[theorem]{Corollary}

\newcommand{\be}{\begin{equation}}
\newcommand{\ee}{\end{equation}}
\newcommand\bes{\begin{eqnarray}}
\newcommand\ees{\end{eqnarray}}
\newcommand{\bess}{\begin{eqnarray*}}
\newcommand{\eess}{\end{eqnarray*}}

\newcommand\eps{\varepsilon}

\def\R{\mathbb{R}}
\def\N{\mathbb{N}}
\def\M{\mathcal{M}}

\begin{document}

\begin{center} {\bf\Large Ground states of planar Schr\"{o}dinger-Poisson systems

                  with an unbounded potential
\footnote{This work was supported by the National Natural Science Foundation of China (No. 11901284),
the Natural Science Foundation of Jiangsu Province (No. BK20180814),
the China Postdoctoral Science Foundation funded project (No. 2020M671531).}
%and Jiangsu Planned Projects for Postdoctoral Research Funds (No. 2019K097).}
}\\[4mm]
 { \large Miao Du$^{\textrm{a},*}$, \ \ Jiaxin Xu$^{\textrm{a}}$}\\[4mm]
{\small
$^{\textrm{a}}$  School of Applied Mathematics, Nanjing University of Finance and Economics,

     Nanjing 210023, People's Republic of China
%$^{\textrm{b}}$ School of Mathematical Sciences, Nanjing Normal University, Nanjing 210023, China
}
\renewcommand{\thefootnote}{}
\footnote{\hspace{-2ex}$^{\ast}$ Corresponding author}
\footnote{\hspace{-2ex}\emph {~E-mail addresses:}
 dumiaomath@163.com, xujiaxin991226@163.com.}
\end{center}

\setlength{\baselineskip}{16pt}%{\setlength\arraycolsep{2pt}

\begin{quote}
  \noindent {\bf Abstract:} In this paper, we deal with a class of planar Schr\"{o}dinger-Poisson systems, namely,
  $-\Delta u+V(x)u+\frac{\gamma}{2\pi}\bigl(\log(|\cdot|)\ast|u|^{2}\bigr)u=b|u|^{p-2}u\ \text{in}\ \mathbb{R}^{2}$,
  where $\gamma > 0$, $b \geq 0$, $p>2$ and $V \in C(\mathbb{R}^2, \mathbb{R})$ is an unbounded potential function
  with $\inf_{\mathbb{R}^2} V >0$. Suppose moreover that the potential $V$ satisfies
  $\left|\{x \in \mathbb{R}^2:\: V(x)\leq M\}\right| < \infty$ for every $M>0$,
  we establish the existence of ground state solutions for this system via variational methods.
  Furthermore, we also explore the minimax characterization of ground state solutions.
  Our main results can be viewed as a counterpart of the results in \cite{Molle-Sardilli-2022},
  where the authors studied the existence of ground state solutions for the above
  planar Schr\"{o}dinger-Poisson system in the case where $b>0$ and $p >4$.


\end{quote}

\section{Introduction}

$ $ \indent
The present paper is devoted to the study of Schr\"{o}dinger-Poisson systems of the type
\begin{equation}\label{eq 1.1}
   \begin{cases}
      i\psi_{t}- \Delta \psi + W(x) \psi + \gamma \phi \psi = b |\psi|^{p-2}\psi
        \ & \text{in}\ \mathbb{R}^{d}\times \mathbb{R},\\
      \Delta \phi= |\psi|^{2}   & \text{in}\ \mathbb{R}^{d}\times \mathbb{R},
   \end{cases}
\end{equation}
where $d \geq 2$,  $\psi: \R^d \times \R \rightarrow \mathbb{C}$ is the time-dependent
wave function, $W: \R^d \rightarrow \R$ is a real external potential,
$\gamma \in \R$ is the coupling coefficient, $b \geq 0$  and $2<p<2^{*}$.
Here, $2^{\ast}$ is the so-called critical Sobolev exponent,
that is, $2^{\ast}=\frac{2d}{d-2}$ if $d \geq 3$ and $2^{\ast} = \infty$ if $d=1$ or $2$.
The function $\phi$ represents an internal potential for a nonlocal self-interaction
of the wave function $\psi$, and the local nonlinear term $b |\psi|^{p-2}\psi$
describes the interaction effect among particles.
System \eqref{eq 1.1} appears in many important physical problems,
such as quantum mechanics (see e.g. \cite{Benguria-1981,Catto-Lions-1993})
and semiconductor theory (see e.g. \cite{Lions-1987,Markowich-1990}).
For more physical interpretations of \eqref{eq 1.1}, we refer the reader
e.g. to \cite{Benci-Fortunato-1998,Mauser-2001}.

One of the most interesting questions about \eqref{eq 1.1} is the existence
of standing (or solitary) wave solutions. By the usual standing wave ansatz
$\psi(x,t)=e^{-i\lambda t}u(x)$, $\lambda\in \R$, problem \eqref{eq 1.1}
is reduced to the following stationary Schr\"{o}dinger-Poisson system:
\begin{equation}\label{eq 1.2}
  \begin{cases}
    -\Delta u + V(x) u + \gamma \phi u = b|u|^{p-2} u \ &\text{in}\ \R^{d},\\
    \Delta \phi= u^{2}   &\text{in}\ \R^{d},
  \end{cases}
\end{equation}
where $V(x)=W(x)+\lambda$. The second equation in \eqref{eq 1.2} determines
$\phi: \mathbb{R}^{d}\rightarrow\mathbb{R}$ only up to harmonic functions.
It is natural to choose $\phi$ as the negative Newton potential of $u^{2}$,
that is, the convolution of $u^{2}$ with the fundamental solution $\Phi_{d}$
of the Laplacian, which is given by
\begin{equation*}
  \Phi_{d}(x)=\frac{1}{2 \pi} \log |x| \quad \text{if} \ d=2
    \qquad \text{and} \qquad
  \Phi_{d}(x)= \frac{1}{d(2-d)\omega_{d}}|x|^{2-d} \quad  \text{if} \ d \geq 3.
\end{equation*}
Here $\omega_{d}$ denotes the volume of the unit ball in $\R^{d}$.
With such a formal inversion of the second equation in \eqref{eq 1.2},
we can derive the integro-differential equation
\begin{equation}\label{eq 1.3}
   -\Delta u + V(x) u + \gamma \left(\Phi_{d}\ast|u|^{2}\right)u  = b|u|^{p-2}u
   \quad \text{in}\ \R^d.
\end{equation}
It is worth noting that, at least formally, \eqref{eq 1.3} has a variational
structure with the associated energy functional
\begin{equation*}
   I_d (u) := \frac{1}{2}\int_{\R^d}\left(|\nabla u|^{2}
     + V(x) u^{2}\right)dx + \frac{\gamma}{4}\int_{\R^d}\int_{\R^d}
     \Phi_d\left(|x-y|\right)u^{2}(x)u^2(y)\,dxdy - \frac{b}{p}\int_{\R^d} |u|^p\,dx.
\end{equation*}
In the case $d \geq 3$, the energy functional $I_d$ is well-defined and of class $C^1$ in $H^1 (\R^d)$,
provided that $V \in L^\infty (\R^d)$. The majority of the literature is centered on the study
of \eqref{eq 1.3} with $d=3$ and $\gamma<0$. In this case, equation \eqref{eq 1.3} and its generalizations
have been extensively investigated over the past few decades, and existence, nonexistence
and multiplicity results of solutions have been obtained via variational methods,
see e.g. \cite{Ambrosetti-Ruiz-2008,Azzollini-Pomponio-2008,Bellazzini-Jeanjean-Luo-2008,
Cerami-2010,DAprile-2004-1,DAprile-2005,Ruiz-2006,Wangjun2012,Wang-Zhou-2007,Zhaoleiga-2013,Zhaoleiga-2008}
and the references therein.

Compared with the case $d \geq 3$, the logarithmic integral kernel
$\Phi_2(x)=\frac{1}{2 \pi} \log |x|$ is sign-changing and presents singularities as
$|x|$ goes to zero and infinity, and therefore the associated energy functional $I_2$
seems much more delicate. In particular, the functional $I_2$ is not well-defined
on the natural Hilbert space $H^{1}(\R^2)$ even if $V \in L^\infty(\R^2)$ and $\inf_{\R^2} V>0$
due to the appearance of the convolution term
\begin{equation*}
  \int_{\mathbb{R}^{2}} \int_{\mathbb{R}^{2}}\log \left(|x-y|\right)u^{2}(x)u^2(y)\,dxdy.
\end{equation*}
Consequently, variational approaches for the higher-dimensional case $d \geq 3$ cannot be directly
adapted to the planar case $d=2$. So the rigorous study of the planar Schr\"{o}dinger-Poisson system
remained open for a long time. This is one of the key reasons why much less is known
in the planar case in which \eqref{eq 1.3} becomes
\begin{equation}\label{eq 1.4}
  -\Delta u+ V(x)u + \frac{\gamma}{2 \pi}\left(\,\log\left(|\cdot|\right)\ast|u|^{2}\,\right)u
   = b |u|^{p-2}u  \quad \text{in}\ \R^{2}.
\end{equation}
In \cite{Stubbe-2008}, Stubbe introduced the smaller Hilbert space
\begin{equation*}
    X:=\left\{u \in H^{1}(\mathbb{R}^{2}):\: \int_{\mathbb{R}^{2}}
    \log\left(1+|x|\right)u^{2}\,dx < \infty \right\}
\end{equation*}
endowed with the norm
\begin{equation*}
  \|u\|_X := \left(\int_{\R^2} \left(|\nabla u|^2 + \left[1 + \log \bigl(1 + |x| \bigl)\right]
   u^2 \right)dx\right)^{1/2},
\end{equation*}
which guarantees that the associated energy functional is well-defined and of class $C^1$ on $X$.
When $ V(x) \equiv  \lambda \geq 0$, $\gamma>0$ and $b=0$ in \eqref{eq 1.4}, Stubbe employed
strict rearrangement inequalities to prove that equation \eqref{eq 1.4} has a unique ground state solution
which is a positive spherically symmetric decreasing function. One of the main contributions
in \cite{Stubbe-2008} is to set a suitable workspace for the planar Schr\"{o}dinger-Poisson system,
but some difficulties arise from the following unpleasant facts. First, while the functional
$I_2$ is invariant under every translation which leaves the external potential $V$ invariant,
the norm of $X$ is not translation invariant. Second, the quadratic part of $I_2$ is not coercive on $X$.
These difficulties have been successfully overcome by Cingolani and Weth in \cite{Cingolani-Weth-2016} where
they developed some new ideas and estimates within the underlying space $X$. Moreover, based on
a surprisingly strong compactness condition (modulo translation) for Cerami sequences,
they also detected the various existence results of solutions for \eqref{eq 1.4} with $\gamma>0$,
$b\geq0$ and $p\geq4$ in a periodic setting. Successively, Weth and the first author \cite{Du-Weth-2017}
extended the above results to the rest case $2<p<4$ by exploring the more complicated
underlying functional geometry with a different variational approach.
When $\gamma>0$ and $b=0$, equation \eqref{eq 1.4} is also referred to as the logarithmic Choquard equation
and can be deduced from the Schr\"{o}dinger-Newton equation \cite{Penrose-1996}.
In \cite{Cingolani-Weth-2016}, it turned out that the logarithmic Choquard equation
has a unique (up to translation) positive solution if $V$ is a positive constant.
Further,  Bonheure,  Cingolani and Van Schaftingen \cite{Bonheure-2017} showed
the sharp asymptotics and nondegeneracy of this unique positive solution.
For more related works on the planar Schr\"{o}dinger-Poisson system,
see e.g. \cite{Albuquerque-2021,Alves-Figueiredo-2019,Azzollini-2021,Cassani-2017,
Chen-2020-2,Chen-2020,Chen-2019,Cingolani-Jeanjean-2019,Liu-Zhang-2022,Liu-Zhang-2023-1}.

In all the works mentioned above for the planar Schr\"{o}dinger-Poisson system, we emphasize that
the recovering of compactness relies heavily on the periodicity or symmetry assumption on the potential $V$.
Very recently,  Liu, R\u{a}dulescu and Zhang \cite{Liu-Zhang-2023} studied the existence of positive
ground state solutions for \eqref{eq 1.4} when $\gamma>0$, $V \in C(\R^2, \R)$ exhibits a finite potential well,
and $b|u|^{p-2}u$ is replaced by $f(u)$ which is required to have either a subcritical or a critical
exponential growth in the sense of Trudinger-Moser. A main feature of \cite{Liu-Zhang-2023}
is that neither the finite potential well nor the reaction satisfies any symmetry or periodicity hypotheses.
Meanwhile, Molle and Sardilli \cite{Molle-Sardilli-2022} considered \eqref{eq 1.4} with $\gamma >0$,
$b >0$ and $p \geq 4$ in a nonperiodic and nonsymmetric setting, where the potential $V$ satisfies
\begin{itemize}
  \item [$(V_0)$] $V \in L_{loc}^1(\R^2)$, $V_0:= \inf_{\R^2} V > 0$ and
    $\left|\{x \in \R^2:\: V(x)\leq M\}\right| < \infty$ for every $M>0$.
\end{itemize}
Note that, the functional $I_2$ is not well-defined on $X$ in virtue of $(V_0)$, so that the natural workspace
in \cite{Molle-Sardilli-2022} is the weighted Hilbert space
\begin{equation*}
  E := \left\{u \in X :\: \int_{\mathbb{R}^{2}}V(x)u^2\,dx < \infty \right\}
\end{equation*}
equipped with the norm
\begin{equation*}
  \|u\|_E := \left(\int_{\R^2}\left(|\nabla u|^2 + V(x)u^2 + \log(1 + |x|)u^2\right)\,dx\right)^{1/2}.
\end{equation*}
With the aid of $(V_0)$, the authors established the Cerami compactness condition, and then obtained
the existence of positive ground state solutions for \eqref{eq 1.4} by using a variant of
the mountain pass theorem. Let us remark that, the assumption $ p > 4$ in \cite{Molle-Sardilli-2022}
is necessary to ensure that the mountain pass energy coincides with the ground state energy
by applying the fibering technique. A natural question arises as to whether there exist
ground state solutions for \eqref{eq 1.4} in the case where $2 < p \leq 4$.
To the best of our knowledge, no existence results for \eqref{eq 1.4} have been available for this case.
This gap of information is unpleasant not only from a mathematical point of view, but also since
the case $p=3$ is relevant in 2-dimensional quantum mechanical models, see \cite[p. 761]{Mauser-2001}.

In the present work, we focus on \eqref{eq 1.4} with $\gamma>0$, and by rescaling
we may assume that $\gamma = 1$. More precisely, we are dealing with
system \eqref{eq 1.2}, the associated scalar equation
\begin{equation}\label{eq 1.5}
  - \Delta u + V(x)u + \frac{1}{2\pi} \left(\log\left(|\cdot|\right)\ast|u|^{2}\right)u
   = b|u|^{p-2}u  \quad  \text{in}\ \R^{2}
\end{equation}
and the associated energy functional $I:\: E \to \R$ defined by
\begin{equation}\label{eq 1.6}
 I(u)=\frac{1}{2}\int_{\mathbb{R}^{2}}\left(|\nabla u|^{2}
     +V(x)u^{2}\right)dx + \frac{1}{8 \pi}\int_{\mathbb{R}^{2}}
     \int_{\mathbb{R}^{2}}\log \left(|x-y|\right)u^{2}(x)u^2(y)\,dxdy
     -\frac{b}{p} \int_{\mathbb{R}^{2}}|u|^p \,dx.
\end{equation}
In the sequel, by a solution of \eqref{eq 1.5} we always mean a weak solution,
that is, a critical point of $I$. A nontrivial solution $u$ of \eqref{eq 1.5} is
called a ground state solution if $I(u) \leq I(w)$ for every nontrivial solution
$w$ of \eqref{eq 1.5}.
The main purpose of this paper is to provide a counterpart
of the results in \cite{Molle-Sardilli-2022} in the case where $2 < p \leq4$.
In particular, we shall establish the existence of ground state solutions
for \eqref{eq 1.5} in this case. Furthermore, we also explore the minimax
characterization of ground state solutions for \eqref{eq 1.5} in the case where $p \geq 3$.

Our first main result is concerned with the existence of ground state solutions
for \eqref{eq 1.5} in the case where $p \geq 4$. For this we define the Nehari manifold
associated to the functional $I$ as
\begin{equation}\label{eq 1.7}
  \mathcal{N} = \left\{ u \in E \backslash \{0\}: \: I'(u) u = 0\right\}.
\end{equation}

\begin{theorem}\label{th 1.1}
Suppose that $b \geq 0$, $p \geq 4$, and that $(V_0)$ holds. Then, the restriction
of $I$ to $\mathcal{N}$ attains a global minimum, and every minimizer $\bar{u}$
of $I|_{\mathcal{N}}$ is a ground state solution of \eqref{eq 1.5} which obeys
the minimax characterization
\begin{equation*}
  I(\bar{u}) = \inf_{u \in E \backslash \{0\}} \sup_{t>0} I(tu).
\end{equation*}
Moreover, if $V$ is locally H\"{o}lder continuous, then $\bar{u} \in C^2(\R^2)$
is a classical solution of \eqref{eq 1.5} that does not change sign.
\end{theorem}

We remark that, Theorem \ref{th 1.1} shows that \eqref{eq 1.5} has a ground
state solution in the case where $p \geq 4$.
This seems to be the first existence result of ground state solutions
for \eqref{eq 1.5} in the special case $b=0$ or $p=4$. Moreover,
we also show that every ground state solution of \eqref{eq 1.5}
obeys a simple minimax characterization. The proof of Theorem \ref{th 1.1}
is based on the method of Nehari manifold as e.g.
in \cite{Badiale-Serra,Rabinowitz-1992,Szulkin-Weth,Willem-1996},
and it is somewhat standard. Here we merely wish to point out that,
the crucial ingredient of the proof is the fact that the weak limit of Cerami sequences
in the space $H_V$ defined by
\begin{equation*}
  H_V = \left\{u \in H^{1}(\R^2) :\: \int_{\R^2} V(x)u^{2}\,dx < \infty \right\}
\end{equation*}
is not equal to zero, and this fact can be verified by noticing that $H_V$
is compactly embedded into $L^s(\R^2)$ for all $s \geq 2$ (see \cite{Bartsch-Wang-1995}).

Nevertheless, it is not difficult to find that Theorem \ref{th 1.1} fails to hold
in the case where $2<p<4$, see \eqref{eq 3.1} and Lemma \ref{lem 3.1} below.
The key difficulty in this case is the competing nature
of the local and nonlocal superquadratic terms in the functional $I$.
In particular, we note that the nonlinearity $u \mapsto f(u):=|u|^{p-2}u$ with $2<p<4$
does not satisfy the Ambrosetti-Rabinowitz type condition
\begin{equation*}
  0<\mu\int_{0}^{u} f(s)\,ds \leq f(u)u \qquad \text{for all $u\neq0$ with some $\mu>4$},
\end{equation*}
which would mean that Palais-Smale sequences or Cerami sequences are bounded in $H_V$.
Moreover, the fact that the function $f(s)/|s|^{3}$ is not increasing on $(-\infty, 0)$ and $(0, \infty)$
prevents us from using the method of Nehari manifold. It follows that
the boundedness of Palais-Smale sequences or Cerami sequences
becomes a major difficulty in the case where $2<p<4$.

Our second main result is concerned with the existence of ground state solutions for \eqref{eq 1.5}
in the case where $2 < p < 4$. In this case, except for $(V_0)$, we need to add the following condition:
\begin{itemize}
  \item [$(V_1)$] $V \in C(\R^2, \,\R)$ is weakly differentiable, and there exist $\alpha, \beta>0$ such that,
  for a.e. $x \in \R^2$,
  \begin{equation*}
    \left|(\nabla V(x), x)\right| \leq \alpha V(x) \qquad \text{and} \qquad 2V(x)+ (\nabla V(x), x) \geq -\beta.
  \end{equation*}
\end{itemize}
This condition is used to obtain a bounded Palais-Smale sequence for the functional $I$.

\begin{theorem}\label{th 1.2}
Suppose that $b >0$, $ 2< p <4$, and that $(V_0)$ and $(V_1)$ hold. Then \eqref{eq 1.5}
has a ground state solution in $E$.
\end{theorem}

Observe that, Theorems \ref{th 1.1} and \ref{th 1.2} give a complete classification
on the existence of ground state solutions for \eqref{eq 1.5}.
We now give the main idea of the proof of Theorem \ref{th 1.2}.
If we apply the mountain pass theorem directly to the functional $I$,
we may then obtain a Palais-Smale sequence for $I$.
However, as mentioned above, it is very difficult to check the boundedness of
this Palais-Smale sequence. To overcome this obstacle,
inspired by \cite{Chen-2020}, we use Jeanjean's monotonicity trick
\cite{Jeanjean-1999} to construct a bounded Palais-Smale sequence
for $I$. Once this step is taken, we can follow arguments
in \cite{Cingolani-Weth-2016} to pass to a subsequence which converges to
a mountain pass solution of \eqref{eq 1.5}, so that the set $\mathcal K$
of nontrivial solutions of \eqref{eq 1.5} is nonempty. Finally we consider
a sequence $\{u_n\}$ in $\mathcal K$ with $I(u_n) \to  \inf_{\mathcal K} I$
as $n \to \infty$, and in the same way as before we may also pass to a subsequence
which converges to a ground state solution of \eqref{eq 1.5}.

In the case $p \ge 4$ studied in Theorem \ref{th 1.1}, $I$ has a rather simple saddle point structure
with respect to the fibres $\{t u: t>0\} \subset E$, $u \in E \setminus \{0\}$, see Lemma \ref{lem 3.1} below.
In particular, this implies that the ground state energy $c_g$ coincides with the minimum of $I$
on the associated Nehari manifold $\mathcal{N}$ and obeys the simple minimax characterization
$c_g= \inf_{u \in E \setminus \{0\}} \sup_{t > 0} I(tu)$. In case $2<p<4$ since, as already remarked above,
the function $|s|^{p-5}s$ is not increasing on $(-\infty, 0)$ and $(0, \infty)$,
this property is lost. As a consequence, we attempt to give a natural characterization of
the ground state energy as a constrained minimum and as a minimax value in this case.

Our third main result is concerned with the minimax characterization of ground state solutions
for \eqref{eq 1.5} in the case where $3 \leq p < 4$. To this aim, besides $(V_0)$,
we need the following condition:
\begin{itemize}
  \item [$(V_2)$] $V \in C^1(\R^2, \R)$ and if $\mathcal{V}(x) := V(x)-\frac{1}{2}(\nabla V(x),x)$,
   then the function $(0, \infty) \rightarrow \R$, $t \mapsto \mathcal{V}(tx)$ is nondecreasing
   on $(0, \infty)$ for every $x \in \R^2$.
\end{itemize}
This condition is thought of as a monotonicity condition on the external potential $V$,
which plays a fundamental role in recovering the saddle point structure of the functional $I$.

Inspired by \cite{Du-Weth-2017}, we introduce the auxiliary functional $J: E \to \R$
defined as
\begin{align}\label{eq 1.8}
  J(u)=&\int_{\R^2} \left(2 |\nabla u|^2+  \mathcal{V}(x) u^2 -\frac{2b(p-1)}{p}|u|^{p}\right)dx
     -\frac{1}{8 \pi} \left(\int_{\mathbb{R}^{2}} u^{2}\,dx\right)^{2} \nonumber \\
   &+ \frac{1}{2 \pi}\int_{\R^2}\int_{\R^2}\log\left(|x-y|\right)u^{2}(x)u^{2}(y)\,dxdy,
\end{align}
and set
\begin{equation}\label{eq 1.9}
  \mathcal M:= \{u \in E \setminus \{0\}:\: J(u)=0\}.
\end{equation}
In the following, the set $\mathcal{M}$ is referred to as the Nehari-Pohozaev mainfold.
Such a similar set was first proposed by Ruiz \cite{Ruiz-2006} for the study of \eqref{eq 1.3}
with $d=3$.

\begin{theorem}\label{th 1.3}
Suppose that $b \geq0 $, $p \geq 3$, and that $(V_0)$ and $(V_2)$ hold.
Then the restriction of $I$ to $\mathcal{M}$ attains a global minimum,
and every minimizer $\bar{u} \in \mathcal{M}$ of $I|_{\mathcal{M}}$ is a ground state solution
of \eqref{eq 1.5} which does not change sign and obeys the minimax characterization
\begin{equation*}
  I(\bar{u}) = \inf_{u \in E \backslash \{0\}} \sup_{t>0} I(u_t),
\end{equation*}
where $u_t \in E$ is defined by $u_t(x):= t^2u(tx)$ for $u \in E$ and $t > 0$.
\end{theorem}

\begin{remark} \rm
Theorem \ref{th 1.3} yields that \eqref{eq 1.5} has a ground state solution in the case
where $p\geq 3$, and every ground state solution of \eqref{eq 1.5} does not change sign.
Moreover, the ground state energy $c_g$ coincides with the minimum of $I$ on the associated
Nehari-Pohozaev manifold $\mathcal{M}$ and obeys a new minimax characterization
$c_g= \inf_{u \in E \setminus \{0\}} \sup_{t > 0} I(u_t)$.
However, this minimax characterization is lost for the case $2<p <3$, and we could not
find any similar saddle point structure of $I$.
\end{remark}

\begin{remark} \rm
By $(V_0)$ and $(V_2)$, we find that $(\nabla V(x),x) \geq 0$ for all $x \in \R^2$,
see Lemma \ref{lem 5.1} below. Moreover, invoking $(V_2)$ we further obtain
\begin{equation}\label{eq 1.10}
  \mathcal{V}(x) \geq V_0 \qquad \text{and} \qquad
   0 \leq (\nabla V(x),x) \leq  2V(x) \qquad \text{for all $x \in \R^2$,}
\end{equation}
so that $(V_1)$ follows. There are indeed many functions
which satisfy $(V_0)$, $(V_1)$ and $(V_2)$.
Here we present two examples. One example is a coercive potential function given by
$V(x) = 1 + |x|^q $ for $x \in \R^2$. It is easy to verify that
$V$ satisfies $(V_0)$ and $(V_1)$ in the case where $q>0$,
while $V$ satisfies $(V_2)$ in the case where $1 < q \leq 2$.
The other is a non-coercive potential function given by
\begin{equation*}
   V(x)=\left\{
    \begin{array}{ll}
       1,  & \text{if } x \in D_n, \: n=1,2, \cdot\cdot\cdot,  \\[2mm]
       1+ |x|^2, & \text{if } x \in \R^2 \backslash \bigcup\limits_{n=1}^{\infty} D_n,
    \end{array}
  \right.
\end{equation*}
where $D_n := [n, n+\frac{1}{2n}] \times [n, n+\frac{1}{2n}]$ is a square for $n \in \N$.
Clearly, we see that $V$ satisfies $(V_0)$.
\end{remark}

The rest of this paper is organized as follows. In Section \ref{sec 2}, we set up
the variational framework for \eqref{eq 1.5} and present some preliminary results.
In Section \ref{sec 3}, we give the proof of Theorem \ref{th 1.1}
on the existence of ground state solutions to \eqref{eq 1.5} for the case $p \geq 4$.
Section \ref{sec 4} is devoted to the proof of Theorem \ref{th 1.2} on the existence of ground state solutions
to \eqref{eq 1.5} for the case $2< p < 4$. Finally, in Section \ref{sec 5} we complete
the proof of Theorem \ref{th 1.3}.

Throughout this paper, we shall make use of the following notations.
$L^{s}(\mathbb{R}^{2})$ denotes the usual Lebesgue space with the norm $|\cdot|_{s}$ for $1\leq s \leq \infty$.
For any $z \in \mathbb{R}^{2}$ and for any $\rho>0$, $B_{\rho}(z)$ denotes the ball of radius $\rho$ centered at $z$.
$X'$ stands for the dual space of $X$.
As usual, the letters $C$, $C_{1}$, $C_{2}, \cdot\cdot\cdot$ denote positive constants
that can change from line to line. Finally, when taking limits, the symbol $o(1)$ stands for any quantity
that tends to zero.

\section{Preliminaries} \label{sec 2}

\indent

In this section, we shall recall the variational framework for \eqref{eq 1.5}
as in \cite{Cingolani-Weth-2016} and present some useful preliminary results.
Throughout the paper, we always assume that $b \geq 0$, $p >2$, and that
$(V_0)$ holds.
Let $H^{1}(\mathbb{R}^{2})$ be the usual Sobolev space endowed with
the scalar product and norm
\begin{equation*}
  \langle u, v\rangle =\int_{\mathbb{R}^{2}} \left(\nabla u \cdot \nabla v
     + uv \right)dx \qquad \text{and}  \qquad \|u\| =\langle u, u\rangle^{1/2},
\end{equation*}
and let $H_V$ be the Hilbert subspace of $u \in H^1(\R^2)$ under the norm
\begin{equation*}
  \|u\|_V := \left(\int_{\R^2} (|\nabla u|^2 + V(x)u^2)\,dx \right)^{1/2}< \infty.
\end{equation*}
Thanks to $(V_0)$, we have the following compact embedding result
due to \cite{Bartsch-Wang-1995}.

\begin{lemma}[\hspace{-0.05ex}{\cite[Theorem 2.1]{Bartsch-Wang-1995}}]\label{lem 2.1}
The embedding of $H_V$ into $L^s(\R^2)$ is compact for all $s \in [2, \infty)$.
\end{lemma}

We now define, for any measurable function $u: \, \mathbb{R}^{2}\rightarrow \mathbb{R}$,
\begin{equation*}
  |u|_{*} =  \left(\int_{\mathbb{R}^{2}} \log\left(1+|x|\right)u^{2}\,dx\right)^{1/2} \in [0, \infty].
\end{equation*}
As has been mentioned in the introduction, we shall work in the Hilbert space
\begin{equation*}
  E := \left\{u \in H_V :\: |u|_{*}< \infty \right\}
\end{equation*}
equipped with the norm $\|u\|_{E}:= \left(\|u\|_V^{2}+|u|^{2}_{*}\right)^{1/2}$.
Of course, $E = X \cap H_V$. Next, we define the symmetric bilinear forms
\begin{align*}
  (u,v) \mapsto B_{1}(u,v)&= \frac{1}{2 \pi} \int_{\mathbb{R}^{2}}\int_{\mathbb{R}^{2}}
     \log \left(1+|x-y|\right)u(x)v(y)\,dxdy,\\
  (u,v) \mapsto B_{2}(u,v)&= \frac{1}{2 \pi} \int_{\mathbb{R}^{2}}\int_{\mathbb{R}^{2}}
     \log \left(1+\frac{1}{|x-y|}\right)u(x)v(y)\,dxdy,\\
  (u,v) \mapsto B_{0}(u,v)&= B_{1}(u,v)-B_{2}(u,v) = \frac{1}{2 \pi} \int_{\mathbb{R}^{2}}
     \int_{\mathbb{R}^{2}}\log \left(|x-y|\right)u(x)v(y)\,dxdy,
\end{align*}
where the definition is restricted, in each case, to measurable functions
$u, v: \R^2 \rightarrow \R$ such that the corresponding double integral is
well-defined in Lebesgue sense. Then, we define on $E$ the associated functionals
\begin{align*}
  N_{1}(u)& :=B_{1}\left(u^{2}, u^{2}\right)= \frac{1}{2 \pi}
     \int_{\mathbb{R}^{2}}\int_{\mathbb{R}^{2}}\log \left(1+|x-y|\right)u^{2}(x)u^{2}(y)\,dxdy,\\
  N_{2}(u)& :=B_{2}\left(u^{2}, u^{2}\right)= \frac{1}{2 \pi} \int_{\mathbb{R}^{2}}
     \int_{\mathbb{R}^{2}}\log \left(1+\frac{1}{|x-y|}\right)u^{2}(x)u^{2}(y)\,dxdy,\\
  N_{0}(u)& :=B_{0}\left(u^{2}, u^{2}\right)= \frac{1}{2 \pi}\int_{\mathbb{R}^{2}}
     \int_{\mathbb{R}^{2}}\log \left(|x-y|\right)u^{2}(x)u^{2}(y)\,dxdy.
\end{align*}
Using the above notations, we can rewrite the functional $I$ defined by \eqref{eq 1.6}
in the following form:
\begin{equation*}
  I(u) = \frac{1}{2}\|u\|_V^2 + \frac{1}{4} N_0(u) - \frac{b}{p}|u|^p_p.
\end{equation*}
Observe that
\begin{equation*}
\log(1+|x-y|)\leq \log(1+|x|+|y|)\leq \log(1+|x|)+\log(1+|y|) \quad \text{for} \ x, y \in \R^2,
\end{equation*}
we have the estimate
\begin{align*}
  B_{1}(uv, wz) &\leq \frac{1}{2 \pi} \int_{\mathbb{R}^{2}}\int_{\mathbb{R}^{2}}
    \bigl[\log\left(1+|x|\right)+\log\left(1+|y|\right)\bigr]
    \left|u(x)v(x)\right|\left|w(y)z(y)\right|dxdy \notag\\
  &\leq \frac{1}{2 \pi} \left(|u|_{*}|v|_{*}|w|_{2}|z|_{2}
    + |u|_{2}|v|_{2}|w|_{*}|z|_{*}\right)
  \quad \text{for}\ u, v, w, z \in E.
\end{align*}
Since $0< \log (1+r)<r$ for $r>0$, we may deduce from the Hardy-Littlewood-Sobolev inequality
(see \cite[Therorem 4.3]{Lieb-Loss-2001}) that
\begin{equation}\label{eq 2.1}
  \left|B_{2}(u,v)\right| \leq  \frac{1}{2 \pi} \int_{\mathbb{R}^{2}}
   \int_{\mathbb{R}^{2}}\frac{1}{|x-y|}\left|u(x)v(y)\right|dxdy
   \leq C_{0}|u|_{\frac{4}{3}}|v|_{\frac{4}{3}}  \quad
   \text{for}\ u, v \in L^{\frac{4}{3}}(\R^2)
\end{equation}
with a constant $C_{0}>0$, which obviously means that
\begin{equation}\label{eq 2.2}
    \left|N_{2}(u)\right| \leq C_{0} |u|_{\frac{8}{3}}^{4} \qquad
    \text{for}\ u \in L^{\frac{8}{3}}(\R^2).
\end{equation}

We close this section with some useful results from \cite{Cingolani-Weth-2016}.

\begin{lemma}[\hspace{-0.05ex}{\cite[Lemma 2.2]{Cingolani-Weth-2016}}]\label{lem 2.2}
The following properties hold true.

\begin{itemize}
  \item [\rm(i)] The space $E$ is compactly embedded into $L^{s}(\R^2)$
     for all $s \in [2, \infty)$.

  \item [\rm(ii)]  The functionals $N_{0}, N_{1}, N_{2}$ and $I$ are of class
     $C^{1}$ on $E$. Moreover,
     \begin{equation*}
        N'_{i}(u)v=4B_{i}\left(u^{2}, uv\right)\quad \text{for}\ u, v \in E
         \ \text{and}\  i=0, 1,  2.
     \end{equation*}

  \item [\rm(iii)] $N_{2}$ is continuously differentiable on $L^{\frac{8}{3}}(\mathbb{R}^{2})$.

  \item [\rm(iv)] $N_{1}$ is weakly lower semicontinuous on $H^{1}(\mathbb{R}^{2})$.

  \item [\rm(v)] $I$ is weakly lower semicontinuous on $E$.

  \item [\rm(vi)] $I$ is lower semicontinuous on $H^{1}(\mathbb{R}^{2})$.

\end{itemize}
\end{lemma}

\begin{lemma}[\hspace{-0.05ex}{\cite[Lemma 2.1]{Cingolani-Weth-2016}}] \label{lem 2.3}
Let $\{u_{n}\}$ be a sequence in $L^{2}(\mathbb{R}^{2})$
such that $u_{n}\rightarrow u\in L^{2}(\mathbb{R}^{2})\backslash\{0\}$
pointwise a.e. in $\mathbb{R}^{2}$. Moreover, let $\{v_{n}\}$ be
a bounded sequence in $L^{2}(\mathbb{R}^{2})$ such that
\begin{equation*}
    \sup_{n \in \mathbb{N}} B_{1}\left(u_{n}^{2}, v_{n}^{2}\right) < \infty.
\end{equation*}
\vskip -0.1 true cm\noindent
Then there exist $n_{0} \in \mathbb{N}$ and $C>0$ such that
$|v_{n}|_{\ast}<C$ for $n \geq n_{0}$.
If moreover $B_{1}(u_{n}^{2}, v_{n}^{2}) \rightarrow 0$ and $|v_{n}|_{2} \rightarrow 0$
as $n \rightarrow \infty$, then $|v_{n}|_{\ast} \rightarrow 0$ as $n \rightarrow \infty$.
\end{lemma}

\begin{lemma}[\hspace{-0.05ex}{\cite[Lemma 2.6]{Cingolani-Weth-2016}}]\label{lem 2.4}
Let $\{u_{n}\}$, $\{v_{n}\}$ and $\{w_{n}\}$ be bounded sequences in $E$
such that $u_{n} \rightharpoonup u$ weakly in $E$. Then, for every
$z \in E$, we have $B_{1}(v_{n}w_{n}, z(u_{n}-u)) \rightarrow 0$ as
$n\rightarrow\infty$.
\end{lemma}

\section{Proof of Theorem \ref{th 1.1}} \label{sec 3}

$ $
\indent
In this section, we will give the proof of Theorem \ref{th 1.1} on the existence of
ground state solutions for \eqref{eq 1.5} in the case where $ p\geq 4$. In the sequel,
we always assume that $ p\geq 4$. To seek a ground state solution of \eqref{eq 1.5},
we consider the Nehari manifold $\mathcal{N}$ defined in \eqref{eq 1.7}, that is,
\begin{equation*}
  \mathcal{N} = \left\{ u \in E \backslash \{0\} : \: I'(u)u = 0\right\}
  =\left\{u \in E \backslash \{0\} : \: \|u\|_V^2 + N_{0}(u) = b|u|^{p}_{p}\right\}.
\end{equation*}
It is easy to see that every nontrivial critical point of $I$ belongs to $\mathcal{N}$.
If $u\in \mathcal{N}$, then
\begin{equation*}
  I(u) = \frac{1}{4}\|u\|_V^2 + \left(\frac{b}{4} - \frac{b}{p}\right)|u|^{p}_{p},
\end{equation*}
and since $p \geq 4$, it follows that
\begin{equation}\label{eq 3.1}
  I(u) \geq \frac{1}{4}\|u\|_V^2 >0 \qquad \text{for $u \in \mathcal{N}$}.
\end{equation}
Then, we define
\begin{equation*}
  m = \inf_{u \in \mathcal{N}}I(u),
\end{equation*}
and we want to show that $m$ is attained by some $u \in \mathcal{N}$ which is
a critical point of $I$ in $E$, and so a ground state solution of \eqref{eq 1.5}.

To start with, we collect some basic properties of $\mathcal{N}$ and $I$.

\begin{lemma}\label{lem 3.1}
Let $u\in E \setminus\{0\}$, then the function $h_{u}:(0, \infty) \rightarrow \mathbb{R}$,
$h_{u}(t)=I(tu)$ is even and has the following properties.
\begin{enumerate}
  \item[\rm(i)] If
    \begin{equation}\label{eq 3.2}
      N_{0}(u)-b|u|^4_4<0 \quad \text{in case $p=4$} \qquad \text{and} \qquad
      N_{0}(u)<0 \ \, \text{or} \  \, b>0 \quad \text{in case $p>4$},
    \end{equation}
  then there exists a unique $t_u \in (0, \infty)$ such that $t_{u}u \in \mathcal{N}$
  and $I(t_{u}u)= \max\limits_{t > 0} I(tu)$. Moreover, $h'_u(t) > 0$ on $(0,  t_u)$
  and $h'_u(t) < 0$ on $(t_u, \infty)$, and $h_{u}(t) \rightarrow -\infty$ as $t \rightarrow \infty$.

  \item[\rm(ii)] If \eqref{eq 3.2} does not hold, then $h'_u(t) >0$ on $(0, \infty)$,
  and $h_{u}(t)\rightarrow \infty$ as $t\rightarrow \infty$.
\end{enumerate}
\end{lemma}

\begin{proof}
A simple computation shows
\begin{equation*}
  \frac{h'_{u}(t)}{t}=\|u\|_V^2 + t^2 N_0(u)- bt^{p-2}|u|^p_p \qquad \text{for $t>0$},
\end{equation*}
so that the desired assertions follow.
\end{proof}

By Lemma \ref{lem 3.1}, we immediately have the following corollary.

\begin{corollary}\label{coro 3.2}
The Nehari manifold $\mathcal{N}$ is not empty and the infimum of $I$ on $\mathcal{N}$
obeys the following minimax characterization:
\begin{equation*}
  \inf_{u\in \mathcal{N}} I(u) = \inf_{u \in E \backslash \{0\}} \sup_{t>0} I(tu).
\end{equation*}
\end{corollary}

Now we are in a position to state that $m > 0$.

\begin{lemma}\label{lem 3.3}
There results $m > 0$.
\end{lemma}

\begin{proof}
By \eqref{eq 2.2} and the Sobolev inequalities, we obtain, for every $u\in \mathcal{N}$,
\begin{equation*}
  \|u\|_V^2 = b|u|^p_p - N_0(u) \leq b|u|^p_p + N_2(u) \leq C_1\|u\|_V^p + C_2\|u\|_V^4.
\end{equation*}
Since $u \neq 0$ and $p>2$, we then have
\begin{equation*}
  \tau := \inf_{u\in \mathcal{N}} \|u\|_V^2 >0.
\end{equation*}
Using \eqref{eq 3.1}, we therefore conclude that
\begin{equation}\label{eq 3.3}
  m = \inf_{u \in \mathcal{N}}I(u) \geq \frac{1}{4} \inf_{u\in \mathcal{N}} \|u\|_V^2
    = \frac{\tau}{4}  >0,
\end{equation}
as claimed.
\end{proof}

Next, we will prove that $m>0$ is attained and every minimizer of $m$ is a critical point of $I$.

\begin{proposition}\label{prop 3.4}
The level $m$ is achieved, and every minimizer of $m$ is a critical point of $I$ in $E$.
\end{proposition}

\begin{proof}
In the following, we divide the proof into two steps.

\emph{Step 1}.  We show that $m$ can be attained. Let $\{u_n\} \subset \mathcal{N}$
be a minimizing sequence for $I$, that is, $I(u_n) \rightarrow m$ as $n \rightarrow\infty$.
By \eqref{eq 3.1}, we know that $\{u_n\}$ is bounded in $H_V$.
Applying the compactness of Lemma \ref{lem 2.1}, up to a subsequence,
there exists $u \in H_V$ such that
\begin{equation}\label{eq 3.4}
  u_n \rightharpoonup u \quad \text{in} \ H_V, \quad
  u_n \rightarrow u \quad \text{in}\ L^s(\R^{2})\ \text{for all}\ s \geq 2, \quad
  u_n(x)\rightarrow u(x) \quad \text{a.e. in}\ \R^2.
\end{equation}
We claim that $u \neq 0$. Arguing by contradiction, we assume that $u=0$. % Assume for contradiction that $u=0$.
It then follows from \eqref{eq 2.2} and \eqref{eq 3.4} that
\begin{equation*}
 \|u_n\|_V^2 + N_1(u_n) = b|u_n|^p_p + N_2(u_n) \leq b|u_n|^p_p + C_0 |u_n|^4_{\frac{8}{3}} = o(1),
\end{equation*}
which means that $\|u_n\|_V \rightarrow 0$ and $N_1(u_n)\rightarrow 0$, so that
\begin{equation*}
  I(u_n) = \frac{1}{2}\|u_n\|_V^2 + \frac{1}{4} \bigl[N_1(u_n) - N_2(u_n)\bigr] - \frac{b}{p}|u_n|^p_p \to 0
  \qquad \text{as $n\rightarrow \infty$.}
\end{equation*}
This gives that $m =0$, which contradicts Lemma \ref{lem 3.3}. So $u \neq 0$, as claimed.

Since $\{u_n\} \subset \mathcal{N}$,  it holds that
\begin{equation*}
  B_1\left(u_n^2, u_n^2\right) = N_1(u_n) =  N_2(u_n) + b|u_n|_p^p - \|u_n\|_V^2,
\end{equation*}
which means that $\sup_{n \in \mathbb{N}} B_1\left(u_n^2, u_n^2\right) < \infty$ due to the boundedness
of $\{\|u_n\|_V\}$. Thus, $|u_n|_\ast$ remains bounded in $n$ by Lemma \ref{lem 2.3}, and therefore
$\{u_n\}$ is bounded in $E$.  Then, passing to a subsequence if necessary, we may assume that
$u_{n} \rightharpoonup u$ in $E$, so that $u \in E$. By the weak lower semicontinuity of the norm,
we further conclude from \eqref{eq 3.4} and Lemma \ref{lem 2.2}(iv) that
\begin{equation}\label{eq 3.5}
  I(u) \leq \liminf_{n \rightarrow \infty} I(u_n) = m,
\end{equation}
and
\begin{equation}\label{eq 3.6}
  \|u\|_V^2 + N_1(u) \leq  N_2(u) + b|u|_p^p.
\end{equation}
If $\|u\|_V^2 + N_1(u) = N_2(u) + b|u|_p^p$, then $u \in \mathcal{N}$, and hence \eqref{eq 3.5} yields that $m$
is achieved at $u$. Since \eqref{eq 3.6} occurs, we only need to deal with the case where
\begin{equation}\label{eq 3.7}
  \|u\|_V^2 + N_1(u) < N_2(u) + b|u|_p^p.
\end{equation}
We now prove that if this happens, it leads to a contradiction. Indeed, it follows from Lemma \ref{lem 3.1}
and \eqref{eq 3.7} that there exists a unique $t \in (0,1)$ such that $tu \in \mathcal{N}$, so that
\begin{align*}
  m&\leq I(tu)=\frac{1}{4} t^2\|u\|_V^2 + \left(\frac{b}{4}-\frac{b}{p}\right)t^p|u|^p_p\\
   &< \frac{1}{4}\|u\|_V^2+ \left(\frac{b}{4}-\frac{b}{p}\right)|u|^p_p\\
   &\leq \liminf_{n \rightarrow \infty}\left[\frac{1}{4}\|u_n\|_V^2 + \left(\frac{b}{4}-\frac{b}{p}\right)|u_n|_p^p\right]\\
   &=\liminf_{n\rightarrow \infty} I(u_n) = m.
\end{align*}
This is impossible, and thus Step 1 is finished.

\emph{Step 2}. We show that every minimizer of $m$ is a critical point of $I$ in $E$.
Let $u$ be an arbitrary minimizer for $m$. We try to prove that $I'(u) v =0$ for all $v \in E$,
and so $u$ is a critical point of $I$.

For every $v \in E$, there exists $\varepsilon >0$ such that $u+s\upsilon \neq 0$ for all $s\in(-\varepsilon, \varepsilon)$.
Define a function $ \varphi: \: (-\varepsilon, \varepsilon) \times (0, \infty) \rightarrow \R$ by
\begin{equation*}
  \varphi(s, t)= t^2\|u + s\upsilon\|_V^2 + t^4 N_0(u + s\upsilon) - bt^p|u + s\upsilon|_p^p.
\end{equation*}
Since $u \in \mathcal{N}$, we see that $\varphi(0,1)=0$. Moreover, $\varphi$ is a $C^1$-function and
\begin{equation*}
  \frac{\partial \varphi}{\partial t}(0, 1) = -2\|u\|_V^2 + (4-p)b|u|^p_p <0.
\end{equation*}
Then, by the implicit function theorem, for $\varepsilon$ small we can determine a $C^1$-function
$t:\:(-\varepsilon, \varepsilon) \rightarrow \mathbb{R}$ such that $t(0)=1$ and
\begin{equation*}
  \varphi(s,t(s))=0 \qquad \text{for all} \ s \in (-\varepsilon, \varepsilon).
\end{equation*}
This also implies that $t(s) > 0$, at least for $\varepsilon$ very small, so that $t(s)(u+sv) \in \mathcal{N}$.
Let us now define $\gamma:\: (-\varepsilon,  \varepsilon) \rightarrow \R$ by
\begin{equation*}
  \gamma(s)=I\left(t(s)(u+s\upsilon)\right),
\end{equation*}
we find that the function $\gamma$ is differentiable and has a minimum point at $s=0$, and therefore
\begin{equation*}
  0 = \gamma'(0) = I'\left(t(0)u\right)\left(t'(0)u+t(0)v\right)
    = t'(0)I'(u)u + I'(u)v =I'(u)v.
\end{equation*}
Since this holds for all $v \in E$, we have $I'(u)=0$. So Step 2 follows.
\end{proof}

The following lemma is the final step in the proof of Theorem \ref{th 1.1}.

\begin{lemma}\label{lem 3.5}
Suppose moreover that $V$ is locally H\"{o}lder continuous. Then, every minimizer $u$
of $m$ is a classical solution of \eqref{eq 1.5} that does not change sign in $\R^2$.
\end{lemma}

\begin{proof}
Since $u \in \mathcal N$ is a minimizer of $I|_{\mathcal N}$, we see that
$|u|$ is also a minimizer of $I|_{\mathcal N}$ due to the fact that
$|u| \in \mathcal{N}$ and $I(u)=I(|u|)$.
Hence, $|u|$ is a critical point of $I$ by Proposition \ref{prop 3.4}.
Using the standard elliptic regularity theory, we have
$|u| \in C^{2}(\mathbb{R}^{2})$ and $-\Delta |u| + q(x) |u| = 0$ in $\mathbb{R}^{2}$
with some function $q(x) \in L_{loc}^{\infty}(\mathbb{R}^{2})$,
see \cite[Proposition 2.3]{Cingolani-Weth-2016} for the details of a similar proof.
Consequently, the strong maximum principle and the fact that $u \neq 0$ imply that
$|u|>0$ in $\mathbb{R}^{2}$, so that $u$ is a classical solution of \eqref{eq 1.5}
and does not change sign.
\end{proof}

The {\em proof of Theorem \ref{th 1.1}} is now completed by merely combining
Corollary \ref{coro 3.2}, Proposition \ref{prop 3.4} and Lemma \ref{lem 3.5}.

\section{Proof of Theorem \ref{th 1.2}} \label{sec 4}

$ $
\indent
In this section, we are devoted to the proof of Theorem \ref{th 1.2} on the existence
of ground state solutions for \eqref{eq 1.5} in the case where $2 < p <4$.
For this purpose, we shall first show the existence of nontrivial solutions
for \eqref{eq 1.5}. Within this step, we employ the following abstract result
due to Jeanjean \cite{Jeanjean-1999}.

\begin{proposition}[\hspace{-0.05ex}{\cite[Theorem 1.1]{Jeanjean-1999}}]\label{prop 4.1}
Let $(H, \, \| \cdot \|_H)$ be a Banach space and $\Lambda \subset \mathbb{R}^{+}$ be an interval.
Consider a family $\{\Phi_{\lambda}\}_{\lambda \in \Lambda}$ of $C^1$-functionals on $H$ of the form
\begin{equation*}
  \Phi_{\lambda}(u) = A(u) - \lambda B(u) \qquad  \text{for} \ \lambda \in \Lambda,
\end{equation*}
where $B(u) \geq 0$ for all $u \in H $, and either $A(u)\rightarrow \infty$
or $B(u)\rightarrow \infty$ as $\|u\|_H \rightarrow \infty$.
Assume that for every $ \lambda \in \Lambda$, there are two points $v_1, v_2$ in $H$ such that
\begin{equation*}
c_{\lambda} := \inf \limits_{\gamma \in \Gamma} \max \limits_{t\in [0,1]} \Phi_{\lambda}(\gamma(t))
 > \max \left\{\Phi_{\lambda}(v_1),\Phi_{\lambda}(v_2)\right\},
\end{equation*}
where
\begin{equation*}
  \Gamma = \left\{\gamma \in C([0,1], H): \: \gamma(0)= v_1 , \ \gamma(1)= v_2\right\}.
\end{equation*}
Then for almost every $\lambda \in \Lambda$, there is a bounded $({\rm PS})_{c_{\lambda}}$ sequence
for $\Phi_{\lambda}$, that is, there exists a sequence $\{u_n(\lambda)\} \subset H$ such that
\begin{enumerate}
  \item[\rm(i)] $\{u_n(\lambda)\}$ is bounded in $H$;

  \item[\rm(ii)] $\Phi_{\lambda}(u_n(\lambda)) \rightarrow c_{\lambda}$;

  \item[\rm(iii)] $\Phi'_{\lambda}(u_n(\lambda)) \rightarrow 0$ in $H'$.
\end{enumerate}
\end{proposition}

Some useful properties of $c_\lambda$ are contained in the following lemma from \cite{Jeanjean-1999}.
\begin{lemma}[\hspace{-0.05ex}{\cite[Lemma 2.3]{Jeanjean-1999}}]\label{lem 4.2}
Under the assumptions of Proposition \ref{prop 4.1}, the map $\lambda  \rightarrow c_\lambda$
is non-increasing and left-continuous.
\end{lemma}

In the following, we always assume that $b > 0$.
As noted above, we shall apply Proposition \ref{prop 4.1} to establish the existence of nontrivial solutions
for \eqref{eq 1.5}. To this aim, inspired by \cite{Chen-2020}, we consider a family of functionals on $E$
defined as
\begin{equation*}
  I_{\lambda}(u) = \frac{1}{2} \|u\|_E^2 + \frac{1}{4} N_0(u)- \lambda \left(\frac{1}{2} |u|_{\ast}^2
   + \frac{b}{p} |u|_p^p \right),
\end{equation*}
for $\lambda \in [1/2, 1]$. Let
\begin{equation}\label{eq 4.1}
  A(u) = \frac{1}{2} \|u\|_E^2 + \frac{1}{4} N_0(u) \qquad \text{and} \qquad
  B(u) = \frac{1}{2} |u|_{\ast}^2 + \frac{b}{p} |u|_p^p,
\end{equation}
then $I_{\lambda}(u) = A(u) - \lambda B(u)$. We need the following lemma.

\begin{lemma}\label{lem 4.3}
We have $B(u) \geq 0$ for all $u \in E$, and either $A(u) \rightarrow \infty$
or $B(u) \rightarrow \infty$ as $\|u\|_E \rightarrow \infty$.
\end{lemma}

\begin{proof}
According to \eqref{eq 4.1}, we know that $B(u) \geq 0$ for all $u \in E$. It thus remains to prove that
either $A(u) \rightarrow \infty$ or $B(u) \rightarrow \infty$ as $\|u\|_E \rightarrow \infty$.
%Arguing by contradiction, we assume that there exists a sequence $\{u_n\}$ such that
Suppose by contradiction that this is false. Then, there exists a sequence $\{u_n\}$ such that
\begin{equation}\label{eq 4.2}
  \|u_n\|_E \rightarrow \infty, \quad \sup_{n \in \mathbb{N}} A(u_n) < \infty
  \quad \text{and} \quad \sup_{n \in \mathbb{N}} B(u_n) < \infty.
\end{equation}
From \eqref{eq 4.1} and \eqref{eq 4.2}, we derive that
\begin{equation}\label{eq 4.3}
  |u_n|_2^2 \leq \pi^{\frac{p-2}{p}}|u_n|_p^2 + \frac{1}{\ln 2} |u_n|_{\ast}^2
   \leq C_1 \qquad \text{for } n \in \N.
\end{equation}
It now follows from \eqref{eq 2.2}, \eqref{eq 4.3} and the Gagliardo-Nirenberg inequality that
\begin{equation*}
  N_2(u_n) \leq C_0 |u_n|_{\frac{8}{3}}^4 \leq C_2 |u_n|_2 ^3|\nabla u_n|_2 \leq C_3 \|u_n\|_E
  \qquad \text{for } n \in \N.
\end{equation*}
Combining this with \eqref{eq 4.1} and \eqref{eq 4.2}, we infer that
\begin{equation*}
   C_4 \geq 4 A(u_n) = 2\|u_n\|_E^2 + N_1(u_n) - N_2(u_n) \geq  2\|u_n\|_E^2 - C_3 \|u_n\|_E
   \qquad \text{for } n \in \N,
\end{equation*}
which means that $\{u_n\}$ is bounded in $E$. This contradicts \eqref{eq 4.2}, and the proof is thus finished.
\end{proof}

The following lemma ensures that the functional $I_\lambda$ possesses the mountain pass geometry.
%The corresponding mountain pass level is also denoted by $c_\lambda$.

\begin{lemma}\label{lem 4.4}
The following properties hold:
\begin{itemize}
  \item [\rm(i)] There exists $v_0 \in E \backslash \{0\}$ independent of $\lambda$ such that
    $I_\lambda(v_0) < 0$ for all $\lambda \in [1/2, 1]$.

  \item [\rm(ii)] $c_{\lambda} := \inf_{\gamma \in \Gamma} \max_{t\in [0,1]}
    I_{\lambda}(\gamma(t)) > \max\{I_{\lambda}(0), \, I_{\lambda}(v_0)\}$ for all $\lambda \in [1/2, 1]$,
    where
    \begin{equation*}
      \Gamma = \left\{\gamma \in C([0,1], E): \: \gamma(0)= 0 , \ \gamma(1)= v_0\right\}.
    \end{equation*}
\end{itemize}
\end{lemma}

\begin{proof}
(i) Let $v \in C^{\infty}_0(\mathbb{R}^{2}) \setminus \{0\}$.
Then we have, for any $t >0$,
\begin{align*}
  I_{1/2}(v_t)
%   =& \frac{t^4}{2} |\nabla v|_2^2 + \frac{t^2}{2} \int_{\R^2} V(t^{-1}x)v^2\,dx
%     +\frac{t^4}{4} N_0(v) - \frac{t^4 \log t}{8\pi} |v|^4_2\\
%   &- \frac{\lambda b}{p} t^{2p-2}|v|_p^p + \frac{(1- \lambda)t^2}{2}\int_{\R^2} \log (1+t^{-1}|x|)v^2\,dx\\
   \leq &\frac{t^4}{2} |\nabla v|_2^2 + \frac{t^2}{2} \int_{\R^2} V(t^{-1}x)v^2\,dx
      + \frac{t^4}{4} N_0(v) - \frac{t^4 \log t}{8 \pi} |v|^4_2 \\
    & - \frac{ b}{2p} t^{2p-2}|v|_p^p + \frac{t^2}{4}|v|_*^2 + \frac{t^2 |\log t|}{4}|v|_2^4.
\end{align*}
This readily implies that $I_{1/2}(v_t) \rightarrow -\infty$ as $t\rightarrow \infty$.
Taking $v_0 := v_t$ for $t$ large enough, we obtain that $ I_{\lambda}(v_0) \leq I_{1/2}(v_0)<0$
for all $\lambda \in [1/2, 1]$, as claimed.

(ii) By \eqref{eq 2.2} and the Sobolev embeddings, we have,  for every $\lambda \in [1/2, 1]$,
\begin{equation*}
  I_\lambda(u) \geq \frac{\|u\|_V^{2}}{2} - \frac{C_{0}}{4}|u|_{\frac{8}{3}}^{4}
   -\frac{b}{p}|u|_{p}^{p} \geq \frac{\|u\|_V^{2}}{2}\left(1-C_{1}\|u\|_V^{2}  -C_{2}\|u\|_V^{p-2}\right)
   \qquad \text{for $u \in E$},
\end{equation*}
which means that $I_\lambda$ has a strict local minimum at $0$, and thus $c_\lambda >0$.
So the assertion follows.
\end{proof}

Next, we present the following compactness property which could be applied to any bounded
(PS) sequence for $I_\lambda$.

\begin{proposition}\label{prop 4.5}
Let $\lambda \in [1/2, 1]$. Then every bounded $({\rm PS})$ sequence for $I_\lambda$
has a convergent subsequence in $E$.
\end{proposition}

\begin{proof}
Let $\{u_n\} \subset E$ be a bounded (PS) sequence for $I_\lambda$, that is,
\begin{equation}\label{eq 4.4}
  \sup_{n \in \N}\|u_n\|_E < \infty, \qquad \sup_{n \in \N} I_{\lambda}(u_n) < \infty
  \qquad \text{and} \qquad I'_{\lambda}(u_n) \rightarrow 0 \quad \text{as $n\rightarrow\infty$}.
\end{equation}
Since $\{u_n\}$ is bounded in $E$, we may assume that, after passing to a subsequence,
there exists $u \in E$ such that $u_{n} \rightharpoonup u$ in $E$.
It then follows from Lemma \ref{lem 2.2}(i) that
$u_{n} \rightarrow u$ in $L^{s}(\mathbb{R}^{2})$ for all $s \geq 2$.
Invoking \eqref{eq 4.4}, we thus deduce that
\begin{align*}
  o(1) =& I_\lambda'(u_{n})(u_{n}-u)\\
     = &\|u_{n}\|_V^{2}-\|u\|_V^{2} + \frac{1}{4}N'_{0}(u_{n}) (u_{n}-u)
       + (1-\lambda) |u_{n}-u|_*^2 + o(1) \\
   \geq &\|u_{n}\|_V^{2}-\|u\|_V^{2}+\frac{1}{4}\left [N'_{1}(u_{n})
        (u_{n}-u)- N'_{2}(u_{n})(u_{n}-u)\right] + o(1),
\end{align*}
where
\begin{equation*}
  \left|\frac{1}{4}N'_{2}(u_{n})(u_{n}-u)\right|
   = \left|B_{2} \left(u_{n}^{2}, \, u_{n}(u_{n}-u)\right)\right|
  \leq C_0 |u_{n}|_{\frac{8}{3}}^{3}\left|u_{n}-u\right|_{\frac{8}{3}}
   \rightarrow 0 \quad \text{as}\ n \rightarrow \infty
\end{equation*}
in view of \eqref{eq 2.1}, and
\begin{equation*}
  \frac{1}{4}N'_{1}(u_{n})(u_{n}-u) = B_{1}\left(u_{n}^{2}, \, u_{n}(u_{n}-u)\right)
   = B_{1}\left(u_{n}^{2}, \, (u_{n}-u)^{2}\right)
    + B_{1}\left(u_{n}^{2}, \, u(u_{n}-u)\right)
\end{equation*}
with
\begin{equation*}
  B_{1}\left(u_{n}^{2}, u(u_{n}-u)\right) \rightarrow 0
    \quad \text{as}\ n\rightarrow\infty
\end{equation*}
according to Lemma \ref{lem 2.4}. Combining these estimates, we derive that
\begin{equation*}
  o(1) \geq \|u_{n}\|_V^{2} - \|u\|_V^{2} + B_{1}\left(u_{n}^{2}, \, (u_{n}-u)^{2}\right) + o(1)
   \geq \|u_{n}\|_V^{2} - \|u\|_V^{2} + o(1),
\end{equation*}
which shows that $\|u_{n}\|_V \rightarrow \|u\|_V$ and $B_{1}\left(u_{n}^{2}, \, (u_{n}-u)^{2}\right)\rightarrow 0$
as $n \rightarrow \infty$. This readily implies that $\|u_{n}-u\|_V \rightarrow 0$, and moreover $|u_{n} - u|_{\ast}\rightarrow 0$
by Lemma \ref{lem 2.3}. We therefore obtain that $\|u_{n} - u \|_{E} \rightarrow 0$ as $n \rightarrow \infty$.
This completes the proof.
\end{proof}

We are now in a position to make use of Proposition \ref{prop 4.1} with $H= E$, $\Lambda = [1/2, 1]$ and $\Phi_\lambda=I_\lambda$.
By Lemmas \ref{lem 4.3} and \ref{lem 4.4}, we see that the functional $I_\lambda$ satisfies the assumptions of Proposition \ref{prop 4.1}.
Therefore, for almost every $\lambda \in [1/2,  1]$, there exists a bounded (PS)$_{c_{\lambda}}$ sequence
for $I_{\lambda}$, and thus Lemma \ref{lem 4.2} and Proposition \ref{prop 4.5} imply that there exists $u_\lambda \in E \backslash \{0\}$  % reveals
such that $I_{\lambda}(u_\lambda) = c_\lambda \in [c_1, \, c_{1/2}]$ and $I'_{\lambda}(u_\lambda) = 0$.
It then follows that there exist two sequences $\{\lambda_n\} \subset [1/2, 1]$ and $\{u_{\lambda_n}\} \subset E$
(for simplicity, we denote $\{u_{\lambda_n}\}$ by $\{u_n\}$) such that
\begin{equation}\label{eq 4.5}
  \lambda_n \rightarrow 1 \quad \text{as $n \rightarrow \infty$}, \qquad
  I_{\lambda_n} (u_n) = c_{\lambda_n} \in [c_1, c_{1/2}] \qquad \text{and} \qquad
  I'_{\lambda_n} (u_n) =0 \quad \text{for $n \in \mathbb{N}$}.
\end{equation}
In the following, we shall show that $\{u_n\}$  is a bounded (PS)$_{c_1}$ sequence for $I = I_1$.
For this purpose,  we need to provide a Pohozaev type identity satisfied by the critical points of $I_\lambda$.

\begin{lemma}\label{lem 4.6}
Suppose that $b \geq 0$, $p>2$, and that $(V_0)$ and $(V_1)$ hold.
Let $u$ be a critical point of $I_{\lambda}$ in $E$ for $\lambda \in [1/2, 1]$,
then we have the following Pohozaev type identity:
\begin{equation*}
\begin{aligned}
 P_{\lambda}(u) :=& \frac{1}{2} \int_{\R^2} \bigl[ 2V(x)+(\nabla V(x), \, x) \bigr]u^2 \,dx
  + N_0(u) + \frac{1}{8 \pi} |u|_2^4 - \frac{2b\lambda}{p}|u|_p^p  \\
 &+ \frac{1- \lambda}{2} \left(2 |u |^2_{\ast} + \int_{\R^2} \frac{|x|}{1+|x|}u^2 \,dx\right)= 0.
\end{aligned}
\end{equation*}
\end{lemma}
\begin{proof}
The proof is standard, so we omit it here, see e.g. \cite[Lemma 2.4]{Du-Weth-2017} for a similar argument.
\end{proof}

From Lemma \ref{lem 4.6}, we now define, for $\lambda \in [1/2, 1]$, the auxiliary functional
$J_\lambda:\: E \rightarrow \R$ as
\begin{equation*}
  J_{\lambda}(u) = 2I'_{\lambda}(u)u - P_{\lambda}(u),
\end{equation*}
then we have
\begin{align}\label{eq 4.6}
 J_{\lambda}(u) =& 2|\nabla u|_2 ^2 + \frac{1}{2}\int_{\R^2} \bigl[2 V(x)-(\nabla V(x), \, x)\bigr]u^2 \,dx
  +  N_0(u)- \frac{1}{8 \pi} |u |_2^4 \notag \\
  &- \frac{2p-2}{p} b\lambda|u|_p^p
   + \frac{1- \lambda}{2} \left(2 |u |^2_{\ast} - \int_{\R^2} \frac{|x|}{1 + |x|}u^2\,dx \right).
\end{align}
Inspired by \cite[Proposition 3.3]{Du-Weth-2017}, we shall establish the following lemma,
which indicates that the sequence $\{u_n\}$ obtained in \eqref{eq 4.5} is bounded in $H_V$.

\begin{lemma}\label{lem 4.7}
Suppose that $b > 0$, $p > 2$, and that $(V_0)$ and $(V_1)$ hold.
Let $\{\lambda_n\} \subset [1/2, 1]$ and $\{u_n\} \subset E$ be two sequences
satisfying \eqref{eq 4.5}, then $\{u_n\}$ is bounded in $H_V$.
\end{lemma}

\begin{proof}
By $(V_1)$, we derive from \eqref{eq 4.5}, \eqref{eq 4.6} and Lemma \ref{lem 4.6} that
\begin{align}\label{eq 4.7}
 c_{1/2} \geq& I_{\lambda_n}(u_n) - \frac{1}{4} J_{\lambda_n}(u_n) \notag\\
    = &\frac{1}{8} \int_{\mathbb{R}^{2}} \bigl [2 V(x) + (\nabla V(x), x)\bigr]u^2_n\,dx
      + \frac{1}{32 \pi} |u_n|^4_2  + \frac{p-3}{2p} b\lambda_n |u_n|^p_p \notag\\
      & + \frac{1- \lambda_n}{8} \left(2|u|^2_{\ast} + \int_{\mathbb{R}^{2}} \frac{|x|}{1+|x|}u_n^2 \,dx\right)\notag\\
  \geq & -\frac{\beta}{8} |u_n|^2_2 + \frac{1}{32 \pi} |u_n|^4_2 + \frac{p-3}{2p} b \lambda_n |u_n|^p_p.
\end{align}
Then we may distinguish the following two cases:

\emph{Case 1: $p>3$}. In this case,  \eqref{eq 4.7} gives that $\{u_{n}\}$ is bounded in $L^{2}(\mathbb{R}^{2})$
and in $L^{p}(\mathbb{R}^{2})$. By \eqref{eq 2.2} and the H\"{o}lder inequality, we then have
\begin{equation*}
  N_2(u_n) \leq C_0 |u_n|^4_{\frac{8}{3}} \leq C_0 |u_n|^{4(1-\theta_0)}_2 |u_n|^{4\theta_0}_p \leq C_1,
\end{equation*}
where $\theta_0 = \frac{p}{4(p-2)}$. Consequently, we can use \eqref{eq 4.5} again to obtain
\begin{equation*}
  2\|u_n\|_V^2 + N_1(u_n) \leq 4I_{\lambda_n}(u_n) + N_2(u_n) + \frac{4b}{p} |u_n|^p_p
  \leq 4c_{1/2} + C_1 + \frac{4b}{p}|u_n|^p_p \leq C_2.
\end{equation*}
This implies that $\{u_n\}$ is bounded in $H_V$.

\emph{Case 2: $p \leq3$}. We first claim that
\begin{equation}\label{eq 4.8}
  |\nabla u_n |_2 \leq C_3  \qquad \text{for $n \in \mathbb{N}$}.
\end{equation}
Suppose to the contrary that \eqref{eq 4.8} does not hold.
We then have, after passing to a subsequence,
\begin{equation*}
  |\nabla u_n|_2 \rightarrow \infty  \qquad \text{as $n \rightarrow \infty$}.
\end{equation*}
Let $t_n:=|\nabla u_n|_2^{-1/2}$ for $n \in \mathbb{N}$, so that $t_n \rightarrow 0$
as $n \rightarrow \infty$. For $n \in \mathbb{N}$, we define the rescaled function
$v_n \in E$ by $v_n(x):= t_n^2 u_n(t_n x)$ for $x \in \R^2$, and a direct calculation
gives that
\begin{equation}\label{eq 4.9}
  |\nabla v_n|_2 =1 \qquad \text{and} \qquad |v_n|_q ^q = t_n^{2q-2} |u_n|_q ^q
  \quad \text{with }  1 \leq q < \infty.
\end{equation}
By the Gagliardo-Nirenberg inequality, we thus deduce that
\begin{equation}\label{eq 4.10}
  |v_n|^p_p \leq C_4 |v_n|^2_2 |\nabla v_n|^{p-2}_2 = C_4 |v_n|^2_2 \qquad \text{for $n \in \mathbb{N}$}.
\end{equation}
Multiplying \eqref{eq 4.7} by $t^4_n$, we can derive from \eqref{eq 4.9} and \eqref{eq 4.10} that
\begin{align*}
  c_{1/2} t^4_n &\geq -\frac{\beta}{8} t_n^4 |u_n|^2_2 + \frac{1}{32 \pi} t_n^4 |u_n|^4_2
      - \frac{b(3-p)}{2p} t_n^4 |u_n|_p^p\\
   &\geq -\frac{\beta}{8} t_n^2 |v_n|^2_2 + \frac{1}{32 \pi} |v_n|^4_2 - \frac{b(3-p)}{2p} C_4 t_n^{6-2p} |v_n|^2_2.
\end{align*}
As a consequence,
\begin{equation}\label{eq 4.11}
  |v_{n}|_{2}=
  \begin{cases}
    o\bigl(t_{n}^{1/2}\bigr)      \qquad &\text{if $p=3$},\\[1.5mm]
    o\bigl(t_{n}^{(3-p)/2}\bigr)  &\text{if $2<p<3$}.
  \end{cases}
\end{equation}
Moreover, by assumption we also obtain
\begin{equation*}
  0 = t^4_n I'_{\lambda_n}(u_n)u_n = t_n^4 \left( \|u_n \|_V ^2 + N_0(u_n)
   - b \lambda_n |u_n|_p^p  + (1-\lambda_n)|u_n|_*^2\right).
\end{equation*}
Combining this with \eqref{eq 4.9}$-$\eqref{eq 4.11} and the fact that
\begin{equation*}
  N_{0}(u_{n})= \frac{t_n^4}{2\pi} \int_{\R^2} \int_{\R^2}
    \log \left(|t_n x-t_n y|\right)u_n^{2}(t_n x)u_n^{2}(t_n y)\:dxdy
  = t_n^{-4}\left( N_0(v_n) + \frac{\log t_n}{2\pi}|v_n|_2^4 \right),
\end{equation*}
we infer that
\begin{equation}\label{eq 4.12}
  0 \geq 1 + N_0(v_n) + \frac{\log t_n}{2 \pi} |v_n|_2 ^4 + o(1)
    = 1 + N_0(v_n) + o(1).
\end{equation}
It then follows from \eqref{eq 2.2}, \eqref{eq 4.11}, \eqref{eq 4.12}
and the Gagliardo-Nirenberg inequality that
\begin{equation*}
  1 \leq 1 + N_1(v_n) \leq N_2(v_n)+ o(1) \leq C_0 |v_n|^4_{\frac{8}{3}}+ o(1) \leq C_5 |v_n|^3_2+ o(1) = o(1),
\end{equation*}
which yields a contradiction, and so \eqref{eq 4.8} holds. Using the Gagliardo-Nirenberg inequality again,
we further conclude from \eqref{eq 4.7} that
\begin{equation*}
  c_{1/2} \geq -\frac{\beta}{8} |u_n|^2_2 +  \frac{1}{32 \pi}  |u_n|^4_2 - \frac{b(3-p)}{2p} |u_n|^p_p
   \geq \frac{1}{32 \pi} |u_n|^4_2  -  C_6 |u_n|_2^2.
\end{equation*}
Consequently, $\{u_n\}$ is bounded in $L^2(\R^2)$. This, together with \eqref{eq 4.8}, implies that
$\{u_{n}\}$ is bounded in $H^1(\R^2)$. We therefore obtain from \eqref{eq 2.2}, \eqref{eq 4.5}
and the Gagliardo-Nirenberg inequality that $\{u_{n}\}$ is bounded in $H_V$, as claimed.
\end{proof}

In the following key lemma, we shall show that $\{u_n\}$ is a bounded (PS)$_{c_1}$ sequence for $I = I_1$.

\begin{lemma}\label{lem 4.8}
Suppose that $b > 0$, $p > 2$, and that $(V_0)$ and $(V_1)$ hold.
Let $\{\lambda_n\} \subset [1/2, 1]$ and $\{u_n\} \subset E$ be two sequences
satisfying \eqref{eq 4.5}, then $\{u_n\}$ is a bounded $({\rm PS})_{c_1}$ sequence for $I$.
\end{lemma}

\begin{proof}
We first observe from Lemma \ref{lem 4.7} that $\{u_n\}$ is bounded in $H_V$.
Then by Lemma \ref{lem 2.1}, there exists $u \in H_V$ such that,
up to a subsequence,
\begin{equation*}
  u_n \rightharpoonup u \quad \text{in $H_V$},  \qquad
  u_n \rightarrow u \quad \text{in $L^s(\mathbb{R}^2)$ for all $s \geq 2$}, \qquad
  u_n(x) \rightarrow u(x) \quad \text{a.e. in $\mathbb{R}^2$.}
\end{equation*}
We now claim that $u \neq 0$. Arguing by contradiction, suppose that $u = 0$.
Invoking \eqref{eq 2.2} and \eqref{eq 4.5}, we obtain
\begin{equation*}
  0 = I'_{\lambda_n}(u_n)u_n = \|u_n\|_V^2 + N_1(u_n) + (1-\lambda_n) |u_n|^2_{\ast} + o(1),
\end{equation*}
which means that $\|u_n\|_V \rightarrow 0$, $N_1(u_n) \rightarrow 0$ and $(1-\lambda_n) |u_n|^2_{\ast} \rightarrow 0$
as $n \rightarrow \infty$. It then follows that
\begin{equation*}
  c_1 \leq c_{\lambda_n} = I_{\lambda_n}(u_n)=\frac{1}{2} \|u_n\|_V^2 +
   \frac{1}{4} N_0(u_n) - \frac{b\lambda_n}{p} |u_n|^p_p  + \frac{1-\lambda_n}{2} |u_n|^2_{\ast} = o(1).
\end{equation*}
This contradicts Lemma \ref{lem 4.4}(ii), so that the claim follows.
Using  \eqref{eq 4.5} again, we have
\begin{equation*}
  B_1\left(u_n^2, u_n^2\right) = N_1(u_n) = N_2(u_n) + b|u_n|_p^p - \|u_n\|_V^2 -(1-\lambda_n) |u_n|_*^2,
\end{equation*}
which implies that $\sup_{n \in \mathbb{N}} B_1\left(u_n^2, u_n^2\right) < \infty$ due to the boundedness
of $\{\|u_n\|_V\}$. We thus deduce that $|u_n|_\ast$ remains bounded in $n$ by Lemma \ref{lem 2.3},
and therefore $\{u_n\}$ is bounded in $E$.

To finish the proof of the lemma, it remains to show that $\{u_n\}$ is a $({\rm PS})_{c_1}$ sequence for $I$,
i.e.,
\begin{equation}\label{eq 4.13}
  I(u_n) \rightarrow c_1 \qquad \text{and} \qquad I'(u_n) \rightarrow 0 \quad \text{as } n \rightarrow \infty.
\end{equation}
Since $\{u_n\}$ is bounded in $E$, we derive from \eqref{eq 4.5} and Lemma \ref{lem 4.2} that
\begin{equation}\label{eq 4.14}
  \lim_{n \to \infty} I(u_n) = \lim_{n \to \infty} \left( I_{\lambda_n}(u_n)
   +\frac{b(\lambda_n -1)}{p}|u_n|_p^p + \frac{\lambda_n -1}{2} |u_n|_*^2 \right)
   =\lim_{n \to \infty} c_{\lambda_n} = c_1.
\end{equation}
Similarly, we also have
\begin{equation*}
   \| I'(u_n)\|_{E'} = \sup_{\varphi \in E \setminus \{0\}}
   \frac{| I'(u_n) \varphi |}{\|\varphi\|_E} \leq   C_1  (1- \lambda_n),
\end{equation*}
which means that $I'(u_n) \to 0$ as $n \to \infty$. Combining this with \eqref{eq 4.14}
gives \eqref{eq 4.13}, as claimed.
\end{proof}

\begin{proof}[Proof of Theorem \ref{th 1.2}]
By Proposition \ref{prop 4.5} and Lemma \ref{lem 4.8}, there exists a critical point
$u \in E \backslash \{0\}$ of $I$ such that $I'(u)=0$ and $I(u) =c_1$. In particular,
the set
\begin{equation*}
  \mathcal{K} := \{u \in E \setminus \{0\}: \: I'(u) = 0\}
\end{equation*}
is nonempty. Moreover, it is easy to check that
\begin{equation}\label{eq 4.15}
  \delta:=\inf_{u \in \mathcal{K}} \|u\|_V >0.
\end{equation}
Let $\{u_n\} \subset \mathcal{K}$  be a sequence such that
\begin{equation}\label{eq 4.16}
  I(u_n) \rightarrow c_g := \inf_{u \in \mathcal{K}} I(u)  \in [-\infty, c_1].
\end{equation}
From the definition of $\mathcal{K}$ and Lemma \ref{lem 4.6}, we see that
the sequence $\{u_n\}$ satisfies
\begin{equation}\label{eq 4.17}
  I'(u_n) = 0 \qquad \text{and} \qquad J(u_n)=0 \quad \text{for all $n \in \mathbb{N}$},
\end{equation}
where $J(u):= J_1(u)$ is given by \eqref{eq 4.6}.
Similar to the arguments in Lemmas \ref{lem 4.7} and \ref{lem 4.8},
we derive from \eqref{eq 4.15}$-$\eqref{eq 4.17} that $\{u_n\}$ is bounded in $E$.
By Proposition \ref{prop 4.5}, we obtain that there exists $u_0 \in E$ such that,
up to a subsequence,
\begin{equation*}
   u_{n} \rightarrow u_0 \quad \text{in} \ E \ \text{as}\ n \rightarrow \infty.
\end{equation*}
It then follows from \eqref{eq 4.15} and \eqref{eq 4.17} that
\begin{equation*}
  \|u_0\|_V \ge \delta \qquad \text{and} \qquad I'(u_0) = 0,
\end{equation*}
so that $u_0 \in \mathcal{K}$. Furthermore, we also have
\begin{equation*}
  I(u_0) = \lim_{n \rightarrow \infty} I(u_n)=c_g,
\end{equation*}
which implies in particular that $c_g > -\infty$.
This completes the proof.
\end{proof}

\section{Proof of Theorem \ref{th 1.3}} \label{sec 5}

$ $
\indent
In this section, we will prove Theorem \ref{th 1.3} on the minimax characterization of
ground state solutions for \eqref{eq 1.5} in the case where $p \geq 3$. In the sequel,
we always assume that $b \geq 0$, $p \geq 3$, and that $(V_0)$ and $(V_2)$ hold.
We start with some elementary observations.

\begin{lemma}\label{lem 5.1}
Suppose that $(V_0)$ and $(V_2)$ hold. Then we have
\begin{equation}\label{eq 5.1}
  (\nabla V(x),x) \geq 0 \qquad \text{for all $x\in \R^2$.}
\end{equation}
\end{lemma}

\begin{proof}
For any fixed $x\in \R^2$,  define $f: (0, \infty) \rightarrow \R$ by
\begin{equation*}
  f(t) = t^2 V(x) - t^2 V(t^{-1}x) + \frac{1-t^2}{2}(\nabla V(x),x).
\end{equation*}
With an easy computation, we see that
\begin{equation*}
  f'(t) = 2t \left( \mathcal{V}(x) - \mathcal{V}(t^{-1}x)\right) \qquad \text{for $t>0$}.
\end{equation*}
It then follows from  $(V_2)$ that $f'(t) \leq 0$ on $(0, 1)$ and $f'(t) \geq 0$ on $(1,\infty)$.
This gives that
\begin{equation*}
  f(t) \geq f(1)=0 \qquad \text{for $t>0$,}
\end{equation*}
and therefore
\begin{equation}\label{eq 5.2}
  t^2 V(x) + \frac{1-t^2}{2}(\nabla V(x),x) \geq  t^2 V(t^{-1}x) \geq 0  \qquad \text{for $t>0$}
\end{equation}
in virtue of $(V_0)$. By passing to the limit $t \rightarrow 0^+$ in \eqref{eq 5.2}, we arrive at \eqref{eq 5.1},
as claimed.
\end{proof}

\begin{lemma}\label{lem 5.2}
Suppose that $\beta: (0, \infty) \to \R$ is a $C^1$-function, and that $t \in (0, \infty) \mapsto t^{-1} \beta'(t)$
is a non-increasing function with a positive lower bound.
Let $C_{i} \in \R$ for $i=1,2,3,4$, and let $C_1, C_3 >0$ and $C_4 \geq 0$. If $p \geq 3$, then the function
\begin{equation*}
  g:(0,\infty) \to \R,\qquad g(t) = C_{1}\beta(t) + C_{2}t^{4} - C_{3}t^{4}\log t-C_{4}t^{2p-2}
\end{equation*}
has a unique critical point $t_0 \in (0, \infty)$ such that $g'(t)>0$ for $t<t_0$ and $g'(t) < 0$ for $t>t_0$.
\end{lemma}
\begin{proof}
The proof is elementary, so we omit it.
\end{proof}

Similarly as in \cite{Du-Weth-2017}, we now consider the Nehari-Pohozaev mainfold
$\mathcal{M}$ defined in \eqref{eq 1.9}, i.e.,
\begin{equation*}
  \mathcal{M}=\left\{u \in E \backslash \{0\}: \: J(u)=0\right\},
\end{equation*}
where $J: E \to \R$ is defined in \eqref{eq 1.8}. It is easy to verify that
\begin{equation*}
  J(u)=2I'(u)u-P(u),
\end{equation*}
where $P(u):= P_1(u)$ is given by Lemma \ref{lem 4.6}. As already noted in the introduction,
by Lemma \ref{lem 4.6} we know that every nontrivial critical point of $I$
is contained in $\mathcal{M}$. If $u\in \mathcal{M}$, then
\begin{equation}\label{eq 5.3}
  I(u) = \frac{1}{4}\int_{\R^2} \Bigl[V(x)+\frac{1}{2}(\nabla V(x),x)\Bigr]u^2\,dx
    +\frac{1}{32 \pi}|u|_2^4 + \frac{b(p-3)}{2p}|u|^{p}_{p},
\end{equation}
and since $p \geq 3$, it follows that
\begin{equation}\label{eq 5.4}
  I(u) \geq \frac{1}{32 \pi} |u|_2^4 >0 \qquad \text{for $u \in \mathcal{M}$}
\end{equation}
according to Lemma \ref{lem 5.1}. With a slight abuse of notation, we define
\begin{equation*}
  m = \inf_{u \in \mathcal{M}}I(u),
\end{equation*}
and we shall show that $m$ is attained by some $u \in \mathcal{M}$ which is
a critical point of $I$ in $E$, and thus a ground state solution of \eqref{eq 1.5}.

For $u\in E$ and $t>0$, we define the rescaled function $Q(t,u) \in E$ by $Q(t,u) = u_t$,
that is,
\begin{equation*}
  Q(t,u)(x)= u_t(x)=t^{2}u(tx) \qquad \text{for $x \in \R^2$.}
\end{equation*}
Next, we state some basic properties of $\mathcal{M}$ and $I$.

\begin{lemma}\label{lem 5.3}
Let $u\in E \backslash \{0\}$, then there exists a unique $t_u \in (0, \infty)$ such that
$Q(t_{u}, u) \in \mathcal{M}$ and $I(Q(t_{u},u))= \max_{t>0} I(Q(t,u))$.
Moreover, the map $E \backslash \{0\} \to (0,\infty)$, $u \mapsto t_u$ is continuous.
\end{lemma}

\begin{proof}
For fixed $u \in E \backslash \{0\}$, we define the function $h_u: (0, \infty) \to \R$,
$h_u(t):=I(Q(t,u))$. A simple computation gives
\begin{align*}
  h_u(t)&=\frac{t^{4}}{2}\int_{\mathbb{R}^{2}}|\nabla u|^{2} \,dx
    +\frac{t^{2}}{2}\int_{\mathbb{R}^{2}} V(t^{-1}x)u^{2} \,dx
    +\frac{t^{4}}{8 \pi}\int_{\mathbb{R}^{2}}\int_{\mathbb{R}^{2}}
    \log\left(|x-y|\right)u^{2}(x)u^{2}(y)\,dxdy \notag \\
  &\quad-\frac{t^{4}\log t}{8 \pi}\left(\int_{\mathbb{R}^{2}}| u|^{2}dx\right)^{2}
    -\frac{bt^{2 p-2}}{p}\int_{\mathbb{R}^{2}}|u|^{p}\,dx.
\end{align*}
Consider now the function $\beta: (0, \infty) \rightarrow \R$ defined as
\begin{equation*}
  \beta(t)= t^2 \int_{\mathbb{R}^{2}} V(t^{-1}x)u^{2} \,dx.
\end{equation*}
By $(V_2)$ and Lemma \ref{lem 5.2}, there exists a unique critical point $t_u \in (0, \infty)$ of $h_u$
such that
\begin{equation}\label{eq 5.5}
  h_u'(t)>0 \quad \text{for $t \in (0, t_u)$} \qquad \text{and}\qquad
  h_u'(t)<0 \quad \text{for $t \in (t_u, \infty)$.}
\end{equation}
Since
\begin{equation*}
  h_u'(t)=  \frac{J(Q(t,u))}{t} \qquad \text{for $t>0$,}
\end{equation*}
we obtain that $\max_{t>0} h_u(t)$ is achieved at a unique $t = t_u$ so that
$h_u'(t_u) = 0$ and $Q(t_u, u) \in \mathcal{M}$. Combining \eqref{eq 5.5} with the fact that
the map $E \setminus \{0\} \to \R, \: u \mapsto h_u'(t)$ is continuous for fixed $t>0$,
we also infer that the map $E \backslash \{0\} \to (0,\infty)$, $u \mapsto t_u$
is continuous, as claimed.
\end{proof}

From Lemma \ref{lem 5.3}, we immediately deduce the following corollary.

\begin{corollary}\label{coro 5.4}
The Nehari-Pohozaev manifold $\mathcal{M}$ is not empty, and the infimum of $I$ on $\mathcal{M}$
obeys the following minimax characterization:
\begin{equation*}
  \inf_{u\in \mathcal{M}} I(u) = \inf_{u\in E\backslash \{0\}} \sup_{t>0} I(u_t).
\end{equation*}
\end{corollary}

In the following, we present a general result which will be used later.
\begin{lemma}\label{lem 5.5}
Let $u \in E$. Then we have
\begin{equation*}
  I\left(Q(t, u)\right) \leq I(u) - \frac{1-t^4}{4} J(u)\qquad \text{for all $t>0$}.
\end{equation*}
\end{lemma}

\begin{proof}
For $u \in E$, define the function $\varphi_u: (0, \infty) \rightarrow \R$ as
\begin{equation*}
  \varphi_u (t) = I(u) - I\left(Q(t, u)\right) - \frac{1-t^4}{4} J(u).
\end{equation*}
We can easily check that $\varphi_u (1) = 0$ and
\begin{equation*}
  \varphi'_u (t) = t^3\left[h'_u(1) - \frac{h_u'(t)}{t^3}\right] \qquad \text{for $t>0$.}
\end{equation*}
Combining this with the fact that the function $(0, \infty) \rightarrow \R$, $ t \mapsto \frac{h_u'(t)}{t^3}$
is non-increasing on $(0, \infty)$, we have
\begin{equation*}
  \varphi'_u (t) \leq 0 \quad \text{for $t \in (0, 1)$} \qquad \text{and} \qquad
  \varphi'_u (t) \geq 0 \quad \text{for $t \in (1,\infty)$}.
\end{equation*}
This says that
\begin{equation*}
  \varphi_u(t) \geq \varphi_u(1)=0 \qquad \text{for $t>0$,}
\end{equation*}
and thus the claim follows.
\end{proof}

Similar to Lemma \ref{lem 4.7}, we have the following lemma which shows that
the minimizing sequence of $m$ is bounded in $H_V$.

\begin{lemma}\label{lem 5.6}
Let $\{u_n\} \subset \M$ be a minimizing sequence for $I$.
Then $\{u_n\}$ is bounded in $H_V$ and moreover,
\begin{equation*}
  \sup_{n \in \N} \int_{\R^2}(\nabla V(x),x)u_n^2\,dx < \infty.
\end{equation*}
\end{lemma}

\begin{proof}
From \eqref{eq 5.3}, we see that
\begin{equation}\label{eq 5.6}
  m + o(1) = \frac{1}{8}\int_{\R^2} \bigl[2V(x)+(\nabla V(x),x)\bigr]u_n^2\,dx
    +\frac{1}{32 \pi}|u_n|_2^4 + \frac{b(p-3)}{2p}|u_n|^{p}_{p}.
\end{equation}
Since $p \geq 3$, we further have, by $(V_0)$ and Lemma \ref{lem 5.1},
\begin{equation}\label{eq 5.7}
  \sup_{n \in \N} \int_{\R^2} V(x)u_n^2\, dx < \infty
  \qquad \text{and} \qquad
  \sup_{n \in \N} \int_{\R^2}(\nabla V(x),x)u_n^2\,dx < \infty.
\end{equation}
Next, we distinguish the following two cases:

\emph{Case 1: $p>3$ and $b>0$}. In this case, \eqref{eq 5.6} and \eqref{eq 5.7}
imply that $\{u_{n}\}$ is bounded in $L^{2}(\R^2)$ and in $L^{p}(\mathbb{R}^{2})$.
Using \eqref{eq 2.2} and the H\"{o}lder inequality, we then find that
\begin{equation*}
  N_2(u_n) \leq C_0 |u_n|^4_{\frac{8}{3}} \leq C_0 |u_n|^{4(1-\theta_0)}_2 |u_n|^{4\theta_0}_p \leq C_1,
\end{equation*}
where $\theta_0 = \frac{p}{4(p-2)}$. Consequently, we have
\begin{equation*}
  2\|u_n\|_V^2 + N_1(u_n) = 4m + N_2(u_n) + \frac{4b}{p} |u_n|_p^p + o(1),
\end{equation*}
which means that $\{u_n\}$ is bounded in $H_V$. By \eqref{eq 5.7}, the lemma is proved.

\emph{Case 2: $p=3$ or $b=0$}. We first claim that
\begin{equation}\label{eq 5.8}
  \sup_{n \in \N}|\nabla u_n |_2 < \infty.
\end{equation}
On the contrary, suppose that \eqref{eq 5.8} does not hold.
We then have, up to a subsequence,
\begin{equation*}
  |\nabla u_n|_2 \rightarrow \infty  \qquad \text{as $n \rightarrow \infty$}.
\end{equation*}
Let $t_n:=|\nabla u_n|_2^{-1/2}$ for $n \in \mathbb{N}$, so that $t_n \rightarrow 0$
as $n \rightarrow \infty$. For $n \in \mathbb{N}$, we define the rescaled function
$v_n \in E$ by $v_n(x):= t_n^2 u_n(t_n x)$ for $x \in \R^2$, and
we can easily conclude that
\begin{equation}\label{eq 5.9}
  |\nabla v_n|_2 =1 \qquad \text{and} \qquad |v_n|_q ^q = t_n^{2q-2} |u_n|_q ^q
  \quad \text{with }  1 \leq q < \infty.
\end{equation}
By the Gagliardo-Nirenberg inequality, we further obtain
\begin{equation}\label{eq 5.10}
  |v_n|^p_p \leq C_4 |v_n|^2_2 |\nabla v_n|^{p-2}_2 = C_4 |v_n|^2_2
  \qquad \text{for $n \in \mathbb{N}$}.
\end{equation}
Multiplying \eqref{eq 5.6} by $t^4_n$, we can derive from \eqref{eq 5.9} and \eqref{eq 5.10} that
\begin{equation*}
  m t_n^4 + o(t_n^4) \geq \frac{1}{32 \pi} t_n^4 |u_n|^4_2  = \frac{1}{32 \pi} |v_n|^4_2,
\end{equation*}
and so
\begin{equation}\label{eq 5.11}
  |v_n|_2 = o\bigl(t_{n}^{1/2}\bigr) \qquad \text{as $n \rightarrow \infty$}.
\end{equation}
Moreover, by \eqref{eq 5.7} we also have
\begin{equation*}
  0 = t_n^4 J(u_n) = t_n^4\left( 2|\nabla u_n|_2^2 + N_0(u_n) - \frac{2b(p-1)}{p} |u_n|_p^p \right)+ o(1).
\end{equation*}
Combining this with \eqref{eq 5.9}$-$\eqref{eq 5.11} and the fact that
\begin{equation*}
  N_{0}(u_{n})= \frac{t_n^4}{2\pi} \int_{\R^2} \int_{\R^2}
   \log \left(|t_n x-t_n y|\right)u_n^{2}(t_n x)u_n^{2}(t_n y)\:dxdy
  = t_n^{-4}\left( N_0(v_n) + \frac{\log t_n}{2\pi}|v_n|_2^4 \right),
\end{equation*}
we infer that
\begin{equation*}
  0 =  2 + N_0(v_n) +  \frac{\log t_n}{2 \pi} |v_n|_2^4  - \frac{2b(p-1)}{p} t_n^{6-2p}|v_n|_p^p + o(1)
    = 2 + N_0(v_n) + o(1).
\end{equation*}
It then follows from \eqref{eq 2.2}, \eqref{eq 5.11} and the Gagliardo-Nirenberg inequality that
\begin{equation*}
  2 \leq 2 + N_1(v_n) = N_2(v_n)+ o(1) \leq C_0 |v_n|^4_{\frac{8}{3}}+ o(1) \leq C_5 |v_n|^3_2+ o(1) = o(1).
\end{equation*}
This is a contradiction. So, \eqref{eq 5.8} holds. Using \eqref{eq 5.7} again,
the lemma is thus finished.
\end{proof}

With the aid of Lemma \ref{lem 5.6}, we shall establish the following lemma which shows that $m>0$.
\begin{lemma}\label{lem 5.7}
There results $m > 0$.
\end{lemma}

\begin{proof}
For each $u\in \mathcal{M}$, by \eqref{eq 1.10}, \eqref{eq 2.2} and the Sobolev inequalities we have
\begin{equation*}
  2|\nabla u|_2^2 + V_0 |u|_2^2  \leq \frac{2b(p-1)}{p} |u|_p^p - N_0 (u) + \frac{1}{8 \pi} |u|_2^4
  \leq C_1\|u\|^p + C_2\|u\|^4.
\end{equation*}
Since $u \neq 0$ and $p>2$, we obtain
\begin{equation}\label{eq 5.12}
  \inf_{u\in \mathcal{M}} \|u\|^2 >0.
\end{equation}
From \eqref{eq 5.4}, we see that $m \geq 0$. If $m >0$ occurs, the lemma is thus proved.
Therefore, arguing by contradiction, we suppose that $m=0$. Let $\{u_n\} \subset \mathcal{M}$
be a minimizing sequence for $I$, that is, $I(u_n) \rightarrow 0$ as $n \rightarrow \infty$.
It then follows from \eqref{eq 5.3} and Lemma \ref{lem 5.1} that
\begin{equation*}
  \int_{\R^2} V(x)u_n^2\, dx \rightarrow 0, \qquad
  \int_{\R^2}(\nabla V(x),x)u_n^2\, dx \rightarrow 0 \qquad \text{and} \qquad
  |u_n|_2 \rightarrow 0,
\end{equation*}
and therefore
\begin{equation*}
  N_2(u_n) \rightarrow 0 \qquad \text{and} \qquad
  |u_n|_p \rightarrow 0 \qquad \text{as  $n \rightarrow \infty$}
\end{equation*}
by \eqref{eq 2.2}, Lemma \ref{eq 5.6} and the Gagliardo-Nirenberg inequality.
Since $\{u_n\} \subset \mathcal{M}$, we further obtain
\begin{equation*}
  |\nabla u_n|_2 \rightarrow 0 \qquad \text{and} \qquad
  N_1(u_n) \rightarrow 0 \qquad \text{as  $n \rightarrow \infty.$}
\end{equation*}
Consequently, we have $\|u_n\| \rightarrow 0$ as $n \rightarrow \infty$,  contradicting \eqref{eq 5.12}.
This completes the proof.
\end{proof}

The last step consists in proving that $m>0$ is attained, and every minimizer of $m$
is a critical point of $I$ in the whole space $E$ which does not change sign.

\begin{proposition}\label{prop 5.8}
The level $m$ is achieved, and every minimizer of $m$ is a critical point
of $I$ in $E$ which does not change sign in $\R^2$.
\end{proposition}

\begin{proof}
In the following, we divide the proof into three parts.

(i) We first show that $m$ can be attained. Let $\{u_n\} \subset \mathcal{M}$
be a minimizing sequence for $I$, i.e., $I(u_n) \rightarrow m$ as $n \rightarrow \infty$.
From Lemma \ref{lem 5.6}, we therefore deduce that $\{u_n\}$ is bounded in $H_V$.
Then, by Lemma \ref{lem 2.1} there exists $u \in H_V$ such that, up to a subsequence,
\begin{equation}\label{eq 5.13}
  u_n \rightharpoonup u \quad \text{in} \ H_V, \quad
  u_n \rightarrow u \quad \text{in}\ L^s(\R^{2})\ \text{for all}\ s \geq 2, \quad
  u_n(x)\rightarrow u(x) \quad \text{a.e. in}\ \R^2.
\end{equation}
Now we claim that $u \neq 0$. Suppose, on the contrary, that $u=0$. Then we have
$u_n \to 0$ in $L^s(\R^{2})$ for all $s \geq 2$, so that $N_2(u_n) \to 0$ as $n \to \infty$
in view of \eqref{eq 2.2}. Since $\{u_n\} \subset \mathcal{M}$,
we further obtain from Lemma \ref{lem 5.1} that
%\begin{align*}
%  0 = J(u_n) =& 2|\nabla u_n|_2^2 + \int_{\R^2} V(x)u_n^2\,dx
%    + \frac{1}{2} \int_{\R^2} (\nabla V(x), x)u_n^2\,dx + N_1(u_n) \notag\\
%   & - N_2(u_n) - \frac{1}{8 \pi}|u_n|_2^4 - \frac{2b(p-1)}{p}|u_n|^{p}_{p}
%    \qquad \text{for $n \in \N$.}
%\end{align*}
%This, together with Lemma \ref{lem 5.1}, implies that
\begin{equation*}
  |\nabla u_n|_2 \rightarrow 0, \quad N_1(u_n) \rightarrow 0, \quad
  \int_{\R^2} V(x)u_n^2\,dx \rightarrow 0 \quad \text{and} \quad
  \int_{\R^2} (\nabla V(x), x)u_n^2\,dx \rightarrow 0,
\end{equation*}
so that
\begin{equation*}
  I(u_n) = \frac{1}{2}\|u_n\|_V^2 + \frac{1}{4} \bigl[N_1(u_n) - N_2(u_n)\bigr]
  - \frac{b}{p}|u_n|^p_p \to 0 \qquad \text{as $n\rightarrow \infty$.}
\end{equation*}
Consequently, we have $m =0$, which contradicts Lemma \ref{lem 5.7}. Thus the claim follows.

Since $\{u_n\} \subset \mathcal{M}$, it follows that
\begin{equation*}
  B_1\left(u_n^2, u_n^2\right) = N_1(u_n) =  N_2(u_n) + \frac{2b(p-1)}{p}|u_n|_p^p + \frac{1}{8 \pi} |u_n|_2^4
  - 2|\nabla u_n|_2^2 - \int_{\R^2} \mathcal{V}(x)u_n^2\,dx,
\end{equation*}
which implies that $\sup_{n \in \mathbb{N}} B_1\left(u_n^2, u_n^2\right) < \infty$ due to Lemma \ref{lem 5.6}.
Hence, $|u_n|_\ast$ remains bounded in $n$ according to Lemma \ref{lem 2.2}, and so $\{u_n\}$
is bounded in $E$.  Then, passing to a subsequence if necessary, we may assume that $u_{n} \rightharpoonup u$
in $E$, so that $u \in E$. Applying the weak lower semicontinuity of the norm,
we derive from \eqref{eq 5.13} and Lemma \ref{lem 2.2}(iv) that
\begin{equation}\label{eq 5.14}
  I(u) \leq \liminf_{n \rightarrow \infty} I(u_n) = m,
\end{equation}
and
\begin{equation}\label{eq 5.15}
  J(u) \leq 0.
\end{equation}
If $J(u) = 0$, then $u \in \mathcal{M}$ and thus, \eqref{eq 5.14} indicates that $m$ can be attained at $u$.
Since \eqref{eq 5.15} holds, we only need to treat the case where
\begin{equation}\label{eq 5.16}
  J(u) < 0.
\end{equation}
We now prove that if \eqref{eq 5.16} occurs, it yields a contradiction. Indeed, Lemmas \ref{lem 5.3} and \ref{lem 5.5}
imply that there exists a unique $t_u \in (0,1)$ such that $Q(t_u, u) \in \mathcal{M}$ and
\begin{equation*}
  I\left(Q(t_u, u)\right) < I(u) - \frac{1}{4} J(u),
\end{equation*}
and therefore
\begin{align*}
  m&\leq I\left(Q(t_u, u)\right) < I(u) - \frac{1}{4} J(u) \\
   &=\frac{1}{8}\int_{\R^2} \bigl[2V(x) + (\nabla V(x),x)\bigr]u^2\,dx
    +\frac{1}{32 \pi}|u|_2^4 + \frac{b(p-3)}{2p}|u|^{p}_{p}\\
   &\leq \lim_{n \rightarrow \infty}\left(\frac{1}{8}\int_{\R^2} \bigl[2V(x) + (\nabla V(x),x)\bigr]u_n^2\,dx
    +\frac{1}{32 \pi}|u_n|_2^4 + \frac{b(p-3)}{2p}|u_n|^{p}_{p}\right)\\
   &=\lim_{n\rightarrow \infty} I(u_n) = m.
\end{align*}
This is impossible, and part (i) is thus proved.

(ii) We next show that every minimizer of $m$ is a critical point of $I$.
Let $u \in \mathcal M$ be an arbitrary minimizer for $I$ on $\mathcal M$.
To show that $u$ is a critical point of $I$, we argue by contradiction
and assume that there exists $v \in E$ such that $I'(u)v < 0$.
Since $I$ is a $C^1$-functional on $E$, we can fix $\eps>0$ with
the following property: \medskip

{\em For every $w_i \in E$ with $\|w_i\|_{E}<\eps$, $i=1,2$, and every $\tau \in (0,\eps)$, we have}
\begin{equation*}
  I(u + w_1 + \tau (v + w_2)) \le I(u +w_1)-\eps \tau.
\end{equation*}
Using Lemma \ref{lem 5.3} and the fact that $t_u=1$, we may choose $\tau \in (0,\eps)$
sufficiently small such that for $t^\tau:= t_{u+\tau v}$,
\begin{equation*}
  \|Q(t^\tau,u)- u\|_E < \eps \qquad \text{and}\qquad  \|Q(t^\tau,v)- v\|_E < \eps.
\end{equation*}
Setting $w_1:= Q(t^\tau,u)-u$ and $w_2:= Q(t^\tau,v)-v$, we then obtain $\|w_i\|_{E}<\eps$
for $i=1,2$, so that, by the property above,
\begin{align*}
  I\left(Q(t^\tau,u + \tau v)\right) &= I\left(Q(t^\tau,u) + \tau Q(t^\tau, v)\right)
   = I(u + w_1 + \tau (v+w_2)) \\
  &\le I(u+w_1) - \eps \tau < I(u+w_1) = I(Q(t^\tau,u)) \le I(u)= m.
\end{align*}
Since $Q(t^\tau,u + \tau v) \in \mathcal{M}$, this contradicts the definition of $m$.
So, part (ii) holds.

(iii) We finally show that every minimizer of $m$ does not change sign in $\R^2$.
The proof of this part is similar to that of Lemma \ref{lem 3.5}, so we omit it here.
%If $u$ is a minimizer of $I|_{\mathcal M}$, then $|u|$ is also a minimizer of $I|_{\mathcal M}$
%due to the fact that $I(u)=I(|u|)$ and $J(u)=J(|u|)$.
%Hence, $|u|$ is a critical point of $I$ by the considerations above.
%The standard elliptic regularity theory then implies that
%$|u| \in C^{2}(\mathbb{R}^{2})$ and $-\Delta |u| + q(x) |u| = 0$ in $\mathbb{R}^{2}$
%with some function $q(x) \in L_{loc}^{\infty}(\mathbb{R}^{2})$.
%By the strong maximum principle and using the fact that $u \neq 0$, we find that
%$|u|>0$ in $\mathbb{R}^{2}$, which means that $u$ does not change sign.
%The proof of this proposition is finished.
\end{proof}

The {\em proof of Theorem \ref{th 1.3}} is now completed by combining Corollary \ref{coro 5.4}
and Proposition \ref{prop 5.8}.

%%%%%%%%%%%%%%%%%%%%%%%%%%%%%%%%%%%%%%%%%%%%%%%%%%%%%
\vskip 0.6 true cm

\end{document}